\documentclass[12pt]{amsart}
\usepackage{amsmath,amsfonts,amssymb,amsthm}
\usepackage[shortalphabetic,abbrev,nobysame]{amsrefs}
\usepackage{mathrsfs}
\usepackage[all]{xy}
 \usepackage{relsize}
\usepackage{hyperref}
\usepackage{accents}
\usepackage{color}
\usepackage{soul}
\allowdisplaybreaks

\DeclareMathOperator{\Res}{Res}
\DeclareMathOperator{\Hom}{Hom}
\DeclareMathOperator{\Sing}{Sing}

\DeclareMathOperator{\End}{End}

\begin{document}

\newtheorem{thm}{Theorem}[section]
\newtheorem{prop}[thm]{Proposition}
\newtheorem{coro}[thm]{Corollary}
\newtheorem{conj}[thm]{Conjecture}
\newtheorem{example}[thm]{Example}
\newtheorem{lem}[thm]{Lemma}
\newtheorem{rem}[thm]{Remark}
\newtheorem{hy}[thm]{Hypothesis}
\newtheorem*{acks}{Acknowledgements}
\theoremstyle{definition}
\newtheorem{de}[thm]{Definition}
\newtheorem{ex}[thm]{Example}

\newtheorem{convention}[thm]{Convention}

\newtheorem{bfproof}[thm]{{\bf Proof}}
\xymatrixcolsep{5pc}

\newcommand{\C}{{\mathbb{C}}}
\newcommand{\Z}{{\mathbb{Z}}}
\newcommand{\N}{{\mathbb{N}}}
\newcommand{\Q}{{\mathbb{Q}}}
\newcommand{\te}[1]{\textnormal{{#1}}}
\newcommand{\set}[2]{{
    \left.\left\{
        {#1}
    \,\right|\,
        {#2}
    \right\}
}}
\newcommand{\sett}[2]{{
    \left\{
        {#1}
    \,\left|\,
        {#2}
    \right\}\right.
}}

\newcommand{\choice}[2]{{
\left[
\begin{array}{c}
{#1}\\{#2}
\end{array}
\right]
}}
\def \<{{\langle}}
\def \>{{\rangle}}

\def\({\left(}
\def\){\right)}

\newcommand{\overit}[2]{{
    \mathop{{#1}}\limits^{{#2}}
}}
\newcommand{\belowit}[2]{{
    \mathop{{#1}}\limits_{{#2}}
}}

\newcommand{\wt}[1]{\widetilde{#1}}

\newcommand{\wh}[1]{\widehat{#1}}

\newcommand{\no}[1]{{
    \mathopen{\overset{\circ}{
    \mathsmaller{\mathsmaller{\circ}}}
    }{#1}\mathclose{\overset{\circ}{\mathsmaller{\mathsmaller{\circ}}}}
}}

\newcommand{\nob}[1]{{
    \mathopen{\overset{\bullet}{
    \mathsmaller{\mathsmaller{\bullet}}}
    }{#1}\mathclose{\overset{\bullet}{\mathsmaller{\mathsmaller{\bullet}}}}
}}

\newlength{\dhatheight}
\newcommand{\dwidehat}[1]{%
    \settoheight{\dhatheight}{\ensuremath{\widehat{#1}}}%
    \addtolength{\dhatheight}{-0.45ex}%
    \widehat{\vphantom{\rule{1pt}{\dhatheight}}%
    \smash{\widehat{#1}}}}
\newcommand{\dhat}[1]{%
    \settoheight{\dhatheight}{\ensuremath{\hat{#1}}}%
    \addtolength{\dhatheight}{-0.35ex}%
    \hat{\vphantom{\rule{1pt}{\dhatheight}}%
    \smash{\hat{#1}}}}

\newcommand{\dwh}[1]{\dwidehat{#1}}

\newcommand{\ck}[1]{\check{#1}}

\newcommand{\dis}{\displaystyle}

\newcommand{\pd}[1]{\frac{\partial}{\partial {#1}}}

\newcommand{\pdiff}[2]{\frac{\partial^{#2}}{\partial #1^{#2}}}


\newcommand{\g}{{\frak g}}
\newcommand{\fg}{\g}
\newcommand{\ff}{{\frak f}}
\newcommand{\f}{\ff}
\newcommand{\gc}{{\bar{\g'}}}
\newcommand{\h}{{\frak h}}
\newcommand{\cent}{{\frak c}}
\newcommand{\notc}{{\not c}}
\newcommand{\Loop}{{\mathcal L}}
\newcommand{\G}{{\mathcal G}}
\newcommand{\D}{\mathcal D}
\newcommand{\T}{\mathcal T}
\newcommand{\Free}{\mathcal F}
\newcommand{\Cfk}{\mathcal C}
\newcommand{\nil}{\mathfrak n}
\newcommand{\al}{\alpha}
\newcommand{\alck}{\al^\vee}
\newcommand{\be}{\beta}
\newcommand{\beck}{\be^\vee}
\newcommand{\ssl}{{\mathfrak{sl}}}
\newcommand{\id}{\te{id}}
\newcommand{\rtu}{{\xi}}
\newcommand{\period}{{N}}
\newcommand{\half}{{\frac{1}{2}}}
\newcommand{\quar}{{\frac{1}{4}}}
\newcommand{\oct}{{\frac{1}{8}}}
\newcommand{\hex}{{\frac{1}{16}}}
\newcommand{\reciprocal}[1]{{\frac{1}{#1}}}
\newcommand{\inverse}{^{-1}}
\newcommand{\inv}{\inverse}
\newcommand{\SumInZm}[2]{\sum\limits_{{#1}\in\Z_{#2}}}
\newcommand{\uce}{{\mathfrak{uce}}}
\newcommand{\Rcat}{\mathcal R}


\newcommand{\orb}[1]{|\mathcal{O}({#1})|}
\newcommand{\up}{_{(p)}}
\newcommand{\uq}{_{(q)}}
\newcommand{\upq}{_{(p+q)}}
\newcommand{\uz}{_{(0)}}
\newcommand{\uk}{_{(k)}}
\newcommand{\nsum}{\SumInZm{n}{\period}}
\newcommand{\ksum}{\SumInZm{k}{\period}}
\newcommand{\overN}{\reciprocal{\period}}
\newcommand{\df}{\delta\left( \frac{\xi^{k}w}{z} \right)}
\newcommand{\dfl}{\delta\left( \frac{\xi^{\ell}w}{z} \right)}
\newcommand{\ddf}{\left(D\delta\right)\left( \frac{\xi^{k}w}{z} \right)}

\newcommand{\ldfn}[1]{{\left( \frac{1+\xi^{#1}w/z}{1-{\xi^{#1}w}/{z}} \right)}}
\newcommand{\rdfn}[1]{{\left( \frac{{\xi^{#1}w}/{z}+1}{{\xi^{#1}w}/{z}-1} \right)}}
\newcommand{\ldf}{{\ldfn{k}}}
\newcommand{\rdf}{{\rdfn{k}}}
\newcommand{\ldfl}{{\ldfn{\ell}}}
\newcommand{\rdfl}{{\rdfn{\ell}}}

\newcommand{\kprod}{{\prod\limits_{k\in\Z_N}}}
\newcommand{\lprod}{{\prod\limits_{\ell\in\Z_N}}}
\newcommand{\E}{{\mathcal{E}}}
\newcommand{\F}{{\mathcal{F}}}

\newcommand{\SY}{{\mathcal{S}}}
\newcommand{\Etopo}{{\mathcal{E}_{\te{topo}}}}

\newcommand{\Ye}{{\mathcal{Y}_\E}}

\newcommand{\rh}{{{\bf h}}}
\newcommand{\rp}{{{\bf p}}}
\newcommand{\rrho}{{{\pmb \varrho}}}
\newcommand{\ral}{{{\pmb \al}}}

\newcommand{\comp}{{\mathfrak{comp}}}
\newcommand{\ctimes}{{\widehat{\boxtimes}}}
\newcommand{\ptimes}{{\widehat{\otimes}}}
\newcommand{\ptimeslt}{{
{}_{\te{t}}\ptimes
}}
\newcommand{\ptimesrt}{{\ot_{\te{t}} }}
\newcommand{\ttp}[1]{{
    {}_{{#1}}\ptimes
}}
\newcommand{\bigptimes}{{\widehat{\bigotimes}}}
\newcommand{\bigptimeslt}{{
{}_{\te{t}}\bigptimes
}}
\newcommand{\bigptimesrt}{{\bigptimes_{\te{t}} }}
\newcommand{\bigttp}[1]{{
    {}_{{#1}}\bigptimes
}}

\newcommand{\ot}{\otimes}
\newcommand{\Ot}{\bigotimes}

\newcommand{\affva}[1]{V_{\wh\g}\(#1,0\)}
\newcommand{\saffva}[1]{L_{\wh\g}\(#1,0\)}
\newcommand{\saffmod}[1]{L_{\wh\g}\(#1\)}

\newcommand{\otcopies}[2]{\belowit{\underbrace{{#1}\ot \cdots \ot {#1}}}{{#2}\te{-times}}}

\newcommand{\twotcopies}[3]{\belowit{\underbrace{{#1}\wh\ot_{#2} \cdots \wh\ot_{#2} {#1}}}{{#3}\te{-times}}}


\newcommand{\tar}{{\mathcal{DY}}_0\(\mathfrak{gl}_{\ell+1}\)}
\newcommand{\U}{{\mathcal{U}}}
\newcommand{\V}{{\mathcal{V}}}
\newcommand{\Y}{{\mathcal{Y}}}

\newcommand{\htar}{\mathcal{DY}_\hbar\(A\)}
\newcommand{\hhtar}{\widetilde{\mathcal{DY}}_\hbar\(A\)}
\newcommand{\htarz}{\mathcal{DY}_0\(\mathfrak{gl}_{\ell+1}\)}
\newcommand{\hhtarz}{\widetilde{\mathcal{DY}}_0\(A\)}
\newcommand{\qhei}{\U_\hbar\left(\hat{\h}\right)}
\newcommand{\n}{{\mathfrak{n}}}
\newcommand{\vac}{{{\bf 1}}}
\newcommand{\vtar}{{{
    \mathcal{V}_{\hbar,\tau}\left(\ell,0\right)
}}}

\newcommand{\qtar}{
    \U_q\(\wh\g_\mu\)}
\newcommand{\rk}{{\bf k}}

\newcommand{\hctvs}[1]{Hausdorff complete linear topological vector space}
\newcommand{\hcta}[1]{Hausdorff complete linear topological algebra}
\newcommand{\ons}[1]{open neighborhood system}
\newcommand{\B}{\mathcal{B}}
\newcommand{\rx}{{\bf x}}
\newcommand{\re}{{\bf e}}
\newcommand{\bk}{{\bf k}}
\newcommand{\rphi}{{\boldsymbol{ \phi}}}

\newcommand{\der}{\mathcal D}


\makeatletter
\renewcommand{\BibLabel}{%
    \Hy@raisedlink{\hyper@anchorstart{cite.\CurrentBib}\hyper@anchorend}%
    [\thebib]%
}
\@addtoreset{equation}{section}
\def\theequation{\thesection.\arabic{equation}}
\makeatother \makeatletter


\title[Quantum affine VA]{Double Yangians and lattice quantum vertex algebras}

\author{Fei Kong$^1$}
\address{Key Laboratory of Computing and Stochastic Mathematics (Ministry of Education), School of Mathematics and Statistics, Hunan Normal University, Changsha, China 410081} \email{kongmath@hunnu.edu.cn}
\thanks{$^1$Partially supported by NSF of China (No. 12371027, No. 12471029).}

\author{Haisheng Li}
\address{Department of Mathematical Sciences, Rutgers University, Camden, NJ 08102}
\email{hli@camden.rutgers.edu}


\begin{abstract}
For any simply-laced GCM $A$,  a $\C[[\hbar]]$-algebra $\wh{\mathcal{DY}}(A)$ was introduced in \cite{KL1},
where it was proved that the universal vacuum $\wh{\mathcal{DY}}(A)$-module ${\mathcal{V}}_A(\ell)$ for any fixed level $\ell$
 is naturally an $\hbar$-adic weak quantum vertex algebra.
Let $L$ be the root lattice of $\g(A)$. As the main results of this paper,
we construct an $\hbar$-adic quantum vertex algebra $V_L[[\hbar]]^{\eta}$
as a formal deformation of the lattice vertex algebra $V_L$ and show that every $V_L[[\hbar]]^{\eta}$-module is naturally
a restricted $\wh{\mathcal{DY}}(A)$-module of level one.
For $A$ of finite type, we obtain a realization of $V_L[[\hbar]]^{\eta}$ as a quotient of
the $\hbar$-adic weak quantum vertex algebra ${\mathcal{V}}_A(1)$, giving a characterization of $V_L[[\hbar]]^{\eta}$-modules as
 restricted $\wh{\mathcal{DY}}(A)$-modules of level one.
\end{abstract}
\maketitle

\section{Introduction}
This is a sequel to \cite{KL1} on centrally extended double Yangians and $\hbar$-adic quantum vertex algebras.
For any simply-laced generalized Cartan matrix $A$,
 an associative $\C[[\hbar]]$-algebra $\wh{\mathcal{DY}}(A)$ was introduced in \cite{KL1},
which coincides with the centrally extended double Yangian associated to the simple Lie algebra $\g:=\g(A)$
when $A$ is of finite type. Furthermore, a new current presentation of $\wh{\mathcal{DY}}(A)$ was given
and for any $\ell\in \C$, a universal vacuum $\wh{\mathcal{DY}}(A)$-module
${\mathcal{V}}_A(\ell)$ of level $\ell$ was constructed.
Among the main results,  it was proved that
 ${\mathcal{V}}_A(\ell)$ has a natural $\hbar$-adic weak quantum vertex algebra structure and
a canonical isomorphism between the category of ${\mathcal{V}}_A(\ell)$-modules and
the category of restricted $\wh{\mathcal{DY}}(A)$-modules of level $\ell$ was obtained.

In this current paper, we continue to study the level-one case.
Let $L$ be the root lattice of the Kac-Moody Lie algebra $\g(A)$.
Using a result of \cite{JKLT-22}, we construct an $\hbar$-adic quantum vertex algebra $V_L[[\hbar]]^\eta$
 as a formal deformation of the lattice vertex algebra $V_L$.
 Among the main results, we prove that every $V_L[[\hbar]]^\eta$-module is naturally
 a restricted $\wh{\mathcal{DY}}(A)$-module of level one and $V_L[[\hbar]]^\eta$ is isomorphic to a quotient
 of the $\hbar$-adic weak quantum vertex algebra ${\mathcal{V}}_A(1)$.
 Assuming that $A$ is of finite type, we obtain a more precise description of $V_L[[\hbar]]^\eta$ as a quotient
 of ${\mathcal{V}}_A(1)$ and give a characterization of $V_L[[\hbar]]^\eta$-modules in terms of restricted
 $\wh{\mathcal{DY}}(A)$-modules of level one.

Now, we continue the introduction to mention some of the main technical issues.
Let $L$ be a finite-rank {\em non-degenerate} even lattice. Associated to $L$, we have a vertex algebra $V_L$
(see \cite{Bor}, \cite{FLM}).  In the construction of $V_L$, a key ingredient is the vector space $\h:=\C\otimes_{\Z}L$.
 Let $\eta(\cdot,x):  \h\rightarrow \h\otimes \C((x))[[\hbar]]$ be any linear map with $\eta(\cdot,x)\equiv 0\mod \hbar$.
Then an $\hbar$-adic quantum vertex algebra $V_L[[\hbar]]^{\eta}$ was obtained in \cite{JKLT-22} as a formal deformation of $V_L$.
The situation for the current work is slightly different:
We start with a simply-laced GCM $A$ and then we have a Kac-Moody Lie algebra $\g(A)$
with $\h$ denoting the canonical Cartan subalgebra.  All the roots  of $(\g(A),\h)$ give us an even lattice $L$.
In general, $L$ is possibly degenerate with $\h\ne \C\otimes_{\Z}L$.
Due to this, we need to modify the construction of $V_L$ and the treatment of \cite{JKLT-22}.

As for the construction of $V_L$, we use the slight generalization of $V_L$ in \cite{LX} (cf. \cite{LW}).
Let $\h$ be any finite-dimensional vector space equipped with a non-degenerate symmetric bilinear form $\<\cdot,\cdot\>$
and let $L\subset \h$ be any (possibly degenerate) even lattice.
Using the same construction (of $V_L$ in \cite{FLM}) we get a vertex algebra denoted by $V_{(\h,L)}$.
(In case $\h=\C\otimes_{\Z}L$, we have $V_{(\h,L)}=V_L$.)
Then define $V_L$ to be the vertex subalgebra generated by the twisted group algebra $\C_{\varepsilon}[L]$.
For any linear map $\eta(\cdot,x):  \h\rightarrow \h\otimes \C((x))[[\hbar]]$ with $\eta(\cdot,x)\equiv 0\mod \hbar$,
slightly modifying the construction of \cite{JKLT-22} we get an $\hbar$-adic quantum vertex algebra
$V_{(\h,L)}[[\hbar]]^{\eta}$ as a formal deformation of $V_{(\h,L)}$.
A fact is that the $\C[[\hbar]]$-submodule $V_L[[\hbar]]$ of $V_{(\h,L)}[[\hbar]]^{\eta}$ is
 an $\hbar$-adic quantum vertex subalgebra, which we denote by $V_L[[\hbar]]^{\eta}$.
 To realize the algebra $\wh{\mathcal{DY}}(A)$ on $V_L[[\hbar]]^{\eta}$-modules, we introduce a particular linear map
 $\eta_\Y(\cdot,x): \h_0\rightarrow \h\otimes \C((x))[[\hbar]]$, where $\h_0:=\C\otimes_{\Z}L$ is a subspace of $\h$, and
 then take $\eta(\cdot,x)$ to be any extension of $\eta_\Y(\cdot,x)$ with domain $\h$.

Note that among the defining relations of the algebra $\wh{\mathcal{DY}}(A)$ are
 the vertex-operator Serre relations and  the ``nonlinear''  $\SY$-commutator relations.
To establish the structure of a restricted $\wh{\mathcal{DY}}(A)$-module on $V_L[[\hbar]]^{\eta_\Y}$-modules,
we make use of the vertex-operator iterate formula and an equivalent version of the Serre relation,
which were obtained in \cite{KL1}.

This paper is organized as follows.
In Section 2, we recall some basics of $\hbar$-adic (weak) quantum vertex algebras and modules.
In Section 3, we recall the algebra $\wh{\mathcal{DY}}(A)$ and its universal vacuum module of level $\ell$.
Section 4 is the core, in which we introduce a particular lattice quantum vertex algebra $V_L[[\hbar]]^{\eta_\Y}$
and give the connection of $V_L[[\hbar]]^{\eta_\Y}$-modules with restricted $\wh{\mathcal{DY}}(A)$-modules of level one.

In this paper, we work on $\C$ (complex numbers) and
we use the formal variable notations and conventions as established in \cite{FLM} and \cite{FHL}.
In addition, we use $\Z_+$ and $\N$ for the sets of positive integers and nonnegative integers, respectively.
For a variable $z$, we also use $\partial_{z}$ for the formal differential operator $\frac{\partial}{\partial z}$.
For any vector space $W$, $W((x))$ and $W((x,y))$ denote the spaces of lower truncated integer power series.
On the other hand, $\C(x)$ and $\C(x,y)$ denote the fields of rational functions.
We often use the $\iota$-map $\iota_{x,y}: \C(x,y)\rightarrow \C((x))((y))$
which is the canonical extension of the embedding of the polynomial ring $\C[x,y]$ into the field $\C((x))((y))$.

\section{$\hbar$-adic quantum vertex algebras  and their modules}

In this section, we recall some basic notions and results on $\hbar$-adic nonlocal vertex algebras,
$\hbar$-adic (weak) quantum vertex algebras, and their modules.

\subsection{Topologically free $\C[[\hbar]]$-modules}
Let $\hbar$ be a formal variable.
A $\C[[\hbar]]$-module $V$ is \emph{topologically free} if $V=V_0[[\hbar]]$
for some $\C$-subspace $V_0$ of $V$.
Let $U$ and $V$ be topologically free $\C[[\hbar]]$-modules.
The $\hbar$-adically completed tensor product $U\wh\ot V$ of $U$ and $V$
is defined to be the $\hbar$-adic completion of $U\otimes V$.
If $U=U_0[[\hbar]]$ and $V=V_0[[\hbar]]$, then $U\wh\ot V=(U_0\ot V_0)[[\hbar]]$.

The following basic facts can be either found (cf. \cite{Ka}) or proved straightforwardly:

\begin{lem}\label{tpfree-basics}
(1)  A $\C[[\hbar]]$-module is topologically free
if and only if it is torsion-free, separated, and $\hbar$-adically complete.
(2) If $W$ is a topologically free $\C[[\hbar]]$-module, then $W=U[[\hbar]]$
for any $\C$-subspace $U$ of $W$ such that $W=U\oplus \hbar W$ over $\C$.
(3) If $U$ is a $\C[[\hbar]]$-submodule of a torsion-free and separated $\C[[\hbar]]$-module
and if $U$ itself is $\hbar$-adically complete, then $U$ is topologically free.
\end{lem}

The following result can be found in \cite{Ka}:

\begin{lem}\label{free-top}
Let $W_1,W_2$ be topologically free $\C[[\hbar]]$-modules and let
$\psi:W_1\rightarrow W_2$ be a continuous $\C[[\hbar]]$-module map.
 If the derived $\C$-linear map $\bar{\psi}: W_1/\hbar W_1\rightarrow W_2/\hbar W_2$ is one-to-one (resp. onto),
 then $\psi$ is one-to-one (resp. onto).
\end{lem}

We also have the following simple facts (see \cite{Li-h-adic}; the proof of Proposition 3.7):

\begin{lem}\label{[U]}
Let $U$ be a $\C[[\hbar]]$-submodule of a $\C[[\hbar]]$-module $W$. Set
\begin{align}
[U]=\{ w\in W\ |\ \hbar^n w\in U\ \text{ for some }n\in \N\}.
\end{align}
Then $\left[ [U]\right]=[U]$. On the other hand, if $[U]=U$, then $U\cap \hbar^n W =\hbar^n U$ for $n\in \N$.
\end{lem}

\begin{lem}\label{simple-fact-2}
Let $U$ be a $\C[[\hbar]]$-submodule of a topologically free $\C[[\hbar]]$-module $W$ such that $[U]=U$. Then
$\left[\bar{U}\right]=\bar{U}$, $\overline{\left[\bar{U}\right]}=\left[\bar{U}\right]$, and $\bar{U}$ is
topologically free.
\end{lem}

\begin{lem}\label{simple-fact-3}
Let $U$ and $V$ be topologically free $\C[[\hbar]]$-modules and let $\psi: U\rightarrow V$
be a $\C[[\hbar]]$-module map. Then $[\ker (\psi)]=\ker (\psi)$, $\overline{\ker (\psi)}=\ker (\psi)$, and
$\ker (\psi)$ is topologically free.
\end{lem}

\begin{lem}\label{basic-facts-top-algebra}
Let $A$ be an associative algebra over $\C[[\hbar]]$, which is topologically free as a $\C[[\hbar]]$-module.

1) For any left (right) ideal $J$ of $A$ as a $\C[[\hbar]]$-algebra, $[J]$ is a  left (right) ideal.

2) For any  left (right) ideal $J$ of $A$ such that $[J]=J$, $\overline{J}$ is a  left (right) ideal of $A$
such that $[\overline{J}]=\overline{J}$ and $\overline{J}$ is topologically free.
 \end{lem}

\begin{lem}\label{strong-submodule}
Let $W$ be a topologically free $\C[[\hbar]]$-module and let $U$ be a $\C[[\hbar]]$-submodule such that
$[U]=U$ and $U$ is $\hbar$-adically complete. Then there exist a $\C$-subspace $W_0$ of $W$ and a $\C$-subspace $U_0$ of $W_0$ such that
$W=W_0[[\hbar]]$ and $U=U_0[[\hbar]]$. Furthermore, $W/U=(W_0/U_0)[[\hbar]]$. In particular, $W/U$ is topologically free.
\end{lem}

Let $W$ be a topologically free $\C[[\hbar]]$-module. For a sequence $\{ w_m\}_{m\in \Z}$ in $W$, we write
$\lim_{m\rightarrow \infty}w_m=0$ to signify that for every $n\in \Z_{+}$, there exists $k\in \Z_{+}$ such that
$w_m\in \hbar^nW$ for all $m\ge k$. Define
\begin{align}
W_{\hbar}((x))=\left\{ \sum_{m\in \Z}w(m)x^{-m-1}\in W[[x,x^{-1}]]\ |\ \lim_{m\rightarrow \infty}w(m)=0\right\}.
\end{align}
If $W=W_0[[\hbar]]$ with $W_0$ a vector space over $\C$, then $W_{\hbar}((x))=W_0((x))[[\hbar]]$.

For each $n\ge 1$, we have a natural $\C[[\hbar]]$-module map
\begin{align}
\pi_n:\ (\te{End} W)[[x,x^{-1}]]\longrightarrow (\te{End} (W/\hbar^nW))[[x,x^{-1}]].
\end{align}
Define
\begin{eqnarray}
\quad\quad\quad
 \E_\hbar(W)=\{\psi(x)\in (\te{End} W)[[x,x^{-1}]]\ |\ \pi_n(\psi(x))\in \E(W/\hbar^nW)\ \ \text{ for }n\ge 1\}.
\end{eqnarray}
Then $\E_\hbar(W)$ is a topologically free $\C[[\hbar]]$-module.
More specifically, if $W=W_0[[\hbar]]$, then $\E_\hbar(W)=\E(W_0)[[\hbar]]$.

\subsection{$\hbar$-adic nonlocal vertex algebras and quantum vertex algebras}

The following notion was introduced in \cite{Li-h-adic}:

\begin{de}\label{de:h-adic-nonlocal-va}
An {\em $\hbar$-adic nonlocal vertex algebra} is a topologically free $\C[[\hbar]]$-module $V$,
equipped with a $\C[[\hbar]]$-module map
$Y(\cdot,x):V\to (\End V)[[x,x^{-1}]]$ and a distinguished vector $\vac\in V$,
satisfying the conditions that
$$Y(u,x)v\in V_{\hbar}((x))\quad \text{ for }u,v\in V,$$
$$Y({\bf 1},x)=1_V, \quad Y(v,x){\bf 1}\in V[[x]]\ \ \text{ and }\ \ (Y(v,x){\bf 1})|_{x=0}=v\ \text{ for }v\in V,$$
and that for any $u,v,w\in V$ and for any $n\in \Z_{+}$, there exists $l\in \N$ such that
\begin{align}
(x_0+x_2)^lY(u,x_0+x_2)Y(v,x_2)w\equiv (x_2+x_0)^lY(Y(u,x_0)v,x_2)w
\end{align}
modulo $\hbar^n V[[x_0^{\pm 1},x_2^{\pm 1}]]$.
\end{de}

Immediately from definition we have:

\begin{prop}
 Let $V$ be a topologically free $\C[[\hbar]]$-module equipped with a $\C[[\hbar]]$-module map
$Y(\cdot,x):V\to (\End V)[[x,x^{-1}]]$ and a vector $\vac\in V$. Then $(V,Y,{\bf 1})$ is an $\hbar$-adic nonlocal vertex algebra
if and only if for every positive integer $n$, $(V/\hbar^nV, \overline{Y}_n,\vac)$
is a nonlocal vertex algebra over $\C$, where $\overline{Y}_n(\cdot,x):V/\hbar^nV\to(\End(V/\hbar^nV))[[x,x^{-1}]]$
is the canonical map induced from $Y(\cdot,x)$.
\end{prop}

Let  $V=V_0[[\hbar]]$ be an $\hbar$-adic nonlocal vertex algebra.  Define a $\C[[\hbar]]$-module map
\begin{align}
Y(x_1,x_2): \ V\wh\ot V\wh\ot \C((x))[[\hbar]]\rightarrow (\End V)[[x_1^{\pm 1},x_2^{\pm 1}]]
\end{align}
by
$$Y(x_1,x_2)(u\otimes v\otimes f(x))=f(x_1-x_2)Y(u,x_1)Y(v,x_2)$$
for $u,v\in V_0,\ f(x)\in \C((x))$.
(Recall that $V\widehat{\otimes}V\widehat{\otimes} \C((x))[[\hbar]]=(V_0\otimes V_0\otimes \C((x)))[[\hbar]]$.)

\begin{de}\label{def-hwqva}
{\em An $\hbar$-adic weak quantum vertex algebra} is an $\hbar$-adic nonlocal vertex algebra $V$
which satisfies the $\hbar$-adic {\em $\SY$-locality:} For any $u,v\in V$,
there exists $F(u,v,x)\in V\wh\ot V\wh\ot \C((x))[[\hbar]]$, satisfying the condition that for every $n\in \Z_{+}$,
there is $k\in \N$ such that
\begin{align}
(x_1-x_2)^kY(u,x_1)Y(v,x_2)\equiv (x_1-x_2)^kY(x_2,x_1)(F(u,v,x))
\end{align}
modulo $\hbar^n(\End V)[[x_1^{\pm 1},x_2^{\pm 1}]]$.
\end{de}

\begin{de}
For $A(x_1,x_2),B(x_1,x_2)\in W[[x_1^{\pm 1},x_2^{\pm 1}]]$ with $W$ any $\C[[\hbar]]$-module,
we write $A(x_1,x_2)\sim B(x_1,x_2)$ if for every $n\in \Z_{+}$, there exists $k\in\N$ such that
\begin{align}
  (x_1-x_2)^kA(x_1,x_2)\equiv (x_1-x_2)^kB(x_1,x_2)\   \mod \hbar^n W[[x_1^{\pm 1},x_2^{\pm 1}]].
\end{align}
\end{de}

From definition, for any $\hbar$-adic weak quantum vertex algebra $V$,
there is a $\C[[\hbar]]$-module map $F(x): V\wh\ot V\rightarrow V\wh\ot V\wh\ot \C((x))[[\hbar]]$ such that
\begin{align}
Y(u,x_1)Y(v,x_2)\sim Y(x_2, x_1)F(x)(v\otimes u).
\end{align}
Such a map $F(x)$ is called an {\em $\SY$-locality operator} of $V$.

 We have (see \cite{Li-h-adic}):

 \begin{prop}\label{lem:q-Jacobi}
 Let $V$ be an $\hbar$-adic weak quantum vertex algebra and let $\SY(x)$ be an $\SY$-locality operator of $V$. Then
 \begin{align}
&x_0\inv\delta\!\(\frac{x_1-x_2}{x_0}\)\!  Y(u,x_1)Y(v,x_2)w\\
&\quad -x_0\inv\delta\!\(\frac{x_2-x_1}{-x_0}\)\! Y(x_2,x_1)(\SY(x)(v\ot u))w\nonumber\\
  =\ & x_1\inv\delta\!\(\frac{x_2+x_0}{x_1}\)\!Y(Y(u,x_0)v,x_2)w\nonumber
\end{align}
 for $u,v,w\in V$, which is called the {\em $\SY$-Jacobi identity for the ordered triple $(u,v,w)$.}
 \end{prop}

Note that an immediate consequence of the $\SY$-Jacobi identity above is
\begin{align}
  Y(u,x_1)&Y(v,x_2)w-Y(x_2,x_1)(\mathcal{S}(x)(v\ot u))w\\
  &= \Res_{x_0}x_1\inv\delta\!\(\frac{x_2+x_0}{x_1}\)\!Y(Y(u,x_0)v,x_2)w\nonumber
\end{align}
(the {\em $\mathcal{S}$-commutator formula}) for $u,v,w\in V$.

For any $\hbar$-adic nonlocal vertex algebra $V$, define a $\C[[\hbar]]$-linear operator $\D$ on $V$ by
\begin{align}
\D (v)=v_{-2}{\bf 1}=\left(\frac{d}{dz}Y(v,z){\bf 1}\right)|_{z=0}\quad \text{ for }v\in V.
\end{align}
The following are the basic properties:
\begin{align}
&[\D,Y(v,x)]=Y(\D (v),x)=\frac{d}{dx}Y(v,x),\\
&e^{z\D}Y(v,x)e^{-z\D}=Y(e^{z\D}v,x)=Y(v,x+z),\\
&Y(v,x){\bf 1}=e^{x\D}v\quad \quad \text{ for }v\in V.
\end{align}

A {\em rational quantum Yang-Baxter operator} on a topologically free $\C[[\hbar]]$-module $W$
 is a $\C[[\hbar]]$-module map
\begin{align}
  \SY(x):\ W\wh\ot W\to W\wh\ot W\wh\ot \C((x))[[\hbar]],
\end{align}
satisfying the following \emph{rational quantum Yang-Baxter equation}:
\begin{align}\label{eq:qyb}
  \SY^{12}(x)\SY^{13}(x+z)\SY^{23}(z)=\SY^{23}(z)\SY^{13}(x+z)\SY^{12}(x)
\end{align}
on $W\wh\ot W\wh \ot W$. It is said to be {\em unitary} if
\begin{align}\label{eq:qyb-unitary}
  \SY^{21}(x)\SY(-x)=1,
\end{align}
where $\SY^{21}(x)=\sigma \SY(x)\sigma$ with $\sigma$ denoting the flip operator on $W\wh\ot W$.

The following notion, formulated in \cite{Li-h-adic}, slightly generalizes Etingof-Kazhdan's notion of
quantum vertex operator algebra (see \cite{EK-qva}; cf. Remark \ref{hqva-EKqvoa} below):

\begin{de}\label{h-qva}
An \emph{$\hbar$-adic quantum vertex algebra}
is an $\hbar$-adic nonlocal vertex algebra $V$ equipped with a unitary rational quantum Yang-Baxter operator
$\SY(x)$ on $V$, such that $\SY(x)$ is an $\SY$-locality operator of $V$,
\begin{align}\label{eq:qyb-shift}
  [\mathcal{D}\ot 1,\SY(x)]=-\frac{d}{dx}\SY(x)
\end{align}
(the \emph{shift condition}),  and
\begin{align}\label{eq:qyb-hex-id}
    \SY(x)(Y(z)\ot 1)=(Y(z)\ot 1)\SY^{23}(x)\SY^{13}(x+z)
  \end{align}
 (the \emph{hexagon identity}) on $V\wh\ot V\wh \ot V$.
\end{de}

\begin{rem}\label{hqva-EKqvoa}
{\em Note that a quantum vertex operator algebra in the sense of \cite{EK-qva}
is an $\hbar$-adic quantum vertex algebra $V$ such that $V/\hbar V$ is a vertex algebra.
For a general $\hbar$-adic quantum vertex algebra $V$,
$V/\hbar V$ is a quantum vertex algebra.}
\end{rem}

The following is a reformulation of Proposition 1.11 of \cite{EK-qva}:

\begin{prop}\label{nondegenerate-wqva}
Let $V$ be an $\hbar$-adic weak quantum vertex algebra which is non-degenerate in the sense that
 the nonlocal vertex algebra $V/\hbar V$ is non-degenerate.
Then there exists a unique $\SY$-locality operator
$\SY(x)$. Furthermore, $\SY(x)$ is a unitary rational quantum Yang-Baxter operator and $(V,\SY(x))$
is an $\hbar$-adic quantum vertex algebra.
\end{prop}

 \begin{de}
 Let $V$ be an $\hbar$-adic nonlocal vertex algebra. A {\em $V$-module} is
 a topologically free $\C[[\hbar]]$-module $W$ equipped with a $\C[[\hbar]]$-module map
 $$Y_W(\cdot,x):\ V\rightarrow  (\End W)[[x,x^{-1}]];\quad v\mapsto Y_W(v,x),$$
satisfying the conditions that
$$Y_W(v,x)\in \E_{\hbar}(W)\quad \text{ for }v\in V,$$
 $Y_W({\bf 1},x)=1_W$ and that
the {\em $\hbar$-adic weak associativity} holds: For any $u,v\in V,\ w\in W$ and for every $n\in \Z_{+}$,
there exists $l\in \N$ such that
\begin{align}
(x_0+x_2)^lY_W(u,x_0+x_2)Y_W(v,x_2)w\equiv (x_2+x_0)^lY_W(Y(u,x_0)v,x_2)w
\end{align}
modulo $\hbar^n W[[x_0^{\pm 1},x_2^{\pm 1}]]$.
\end{de}

Let $(W,Y_W)$ be a module for an $\hbar$-adic nonlocal vertex algebra $V=V_0[[\hbar]]$.
As an analogue of the map $Y(x_1,x_2)$, define a $\C[[\hbar]]$-module map
\begin{align}
Y_W(x_1,x_2):\ V\widehat{\otimes}V\widehat{\otimes} \C((x))[[\hbar]]\rightarrow (\End W)[[x_1^{\pm 1},x_2^{\pm 1}]]
\end{align}
by
\begin{align*}
Y_W(x_1,x_2)(u\otimes v\otimes f(x))=f(x_1-x_2)Y_W(u,x_1)Y_W(v,x_2)
\end{align*}
for $u,v\in V_0,\ f(x)\in \C((x))$.

We have (see \cite{Li-h-adic}; cf. \cite{Li-nonlocal}) :

\begin{lem}\label{sim-commutator}
Let $V$ be an $\hbar$-adic nonlocal vertex algebra, let  $(W,Y_W)$ be a $V$-module, and let
$$u,v \in V,\ A(x)\in V\widehat{\otimes}V\widehat{\otimes} \C((x))[[\hbar]].$$
Then
$$Y_W(u,x_1)Y_W(v,x_2)\sim Y_W(x_2,x_1)(A(x))$$
if and only if
 \begin{align}
&x_0^{-1}\delta\!\(\frac{x_1-x_2}{x_0}\)\! Y_W(u,x_1)Y_W(v,x_2)
-x_0^{-1}\delta\!\(\frac{x_2-x_1}{-x_0}\)\! Y_W(x_2,x_1)(A(x))\nonumber\\
 &\hspace{2cm} = x_1^{-1}\delta\!\(\frac{x_2+x_0}{x_1}\)\!Y_W(Y(u,x_0)v,x_2).\nonumber
\end{align}
\end{lem}

The following was also proved in  \cite{Li-h-adic} (Prop. 2.25):

\begin{prop}\label{prop-2.25}
Let $V$ be an $\hbar$-adic nonlocal vertex algebra, let $(W,Y_W)$ be a $V$-module, and let
$$u,v, c^{(0)},c^{(1)},\dots \in V,\ A(x)\in V\widehat{\otimes}V\widehat{\otimes} \C((x))[[\hbar]]$$
with $\lim_{n\rightarrow \infty}c^{(n)}=0$. If
\begin{align*}
Y(u,x_1)Y(v,x_2)-Y(x_2,x_1)(A(x))
=\sum_{n\ge 0}Y(c^{(n)},x_2)\frac{1}{n!}\left(\frac{\partial}{\partial x_2}\right)^nx_1^{-1}\delta\!\left(\frac{x_2}{x_1}\right)
\end{align*}
on $V$, then
\begin{align*}
Y_W(u,x_1)Y_W(v,x_2)-Y_W(x_2,x_1)(A(x))
=\sum_{n\ge 0}Y_W(c^{(n)},x_2)\frac{1}{n!}\left(\frac{\partial}{\partial x_2}\right)^nx_1^{-1}\delta\!\left(\frac{x_2}{x_1}\right)\!.
\end{align*}
Furthermore, the converse is also true if $(W,Y_W)$ is faithful.
\end{prop}

\subsection{Some technical results}
Here, we recall from \cite{KL1} the notions of strong ideal of an $\hbar$-adic nonlocal vertex algbera $V$,
strong submodule of a $V$-module $W$, relatively simple $\hbar$-adic nonlocal vertex algebra, and relatively simple module.
We also recall the results on Serre relation and a vertex-operator iterate formula.

\begin{de}
Let $V$ be an $\hbar$-adic nonlocal vertex algebra. An {\em $\hbar$-adic nonlocal vertex subalgebra} of $V$
is a $\C[[\hbar]]$-submodule $U$ such that $U$ itself is $\hbar$-adically complete,
\begin{eqnarray}\label{subalgebra-closure}
{\bf 1}\in U,\ \ u_mv\in U\quad \text{ for }u,v\in U,\ m\in \Z,
\end{eqnarray}
and $\lim_{m\rightarrow \infty}u_mv=0$ with respect to the $\hbar$-adic topology of $U$.
Furthermore, $U$ is called a {\em strong $\hbar$-adic nonlocal vertex subalgebra} if $U=[U]$ in $V$.
\end{de}

\begin{de}
Let $V$ be an $\hbar$-adic nonlocal vertex algebra. A {\em strong ideal} of $V$ is a $\C[[\hbar]]$-submodule $J$ such that
$[J]=J$, $J$ is $\hbar$-adically complete, and
\begin{align}
v_m J\subset J,\ \ a_m V\subset J\quad \text{ for }v\in V,\ m\in \Z,\ a\in J.
\end{align}
An $\hbar$-adic nonlocal vertex algebra $V$ is said to be {\em relatively simple}
if $V$ is its only nonzero strong ideal.
\end{de}


\begin{lem}\label{kernal-strong-ideal}
Let $\psi: V\rightarrow K$ be a homomorphism of $\hbar$-adic nonlocal vertex algebras.
Then $\ker (\psi)$ is a strong ideal of $V$.
\end{lem}

\begin{lem}\label{strong-ideal}
Let $J$ be a strong ideal of an $\hbar$-adic nonlocal vertex algebra $V$.
Then the $\C[[\hbar]]$-module $V/J$ is topologically free and it is an $\hbar$-adic nonlocal vertex algebra.
\end{lem}

\begin{lem}\label{strong-ideal-by-T}
Let $T$ be a subset of an $\hbar$-adic nonlocal vertex algebra.
Denote by $(T)$ the intersection of all strong ideals of $V$ containing $T$.
Then $(T)$ is the smallest strong ideal containing $T$.
\end{lem}

\begin{prop}\label{simple-hva}
Let $V$ be an $\hbar$-adic nonlocal vertex algebra such that $V/\hbar V$ is a simple nonlocal vertex algebra over $\C$.
Then $V$ is relatively simple.
\end{prop}

Let $W$ be a module for an $\hbar$-adic nonlocal vertex algebra $V$.
A  {\em strong $V$-submodule} of $W$ is a $\C[[\hbar]]$-submodule
$U$ such that $[U]=U$, $U$ is $\hbar$-adically complete, and
\begin{align}
a_m u\subset U\quad \text{ for }a\in V,\ m\in \Z,\ u\in U.
\end{align}
A $V$-module $W$ is said to be {\em relatively simple (irreducible)} if
 $W$ is its only nonzero strong $V$-submodule.

\begin{lem}\label{hva-quotient-module}
Let $W$ be a module for an $\hbar$-adic nonlocal vertex algebra $V$ and let $U$ be a strong $V$-submodule of $W$.
Then $W/U$ is topologically free and it is a $V$-module.
\end{lem}

\begin{prop}\label{simple-hva-module}
Let $V$ be an $\hbar$-adic nonlocal vertex algebra and let $W$ be a $V$-module such that
$W/\hbar W$ is an irreducible $V/\hbar V$-module. Then $W$ is relatively simple.
\end{prop}


\begin{de}\label{def-Serre-relation}
Let $W$ be a topologically free $\C[[\hbar]]$-module and let $a(x),b(x)\in \E_{\hbar}(W)$.
We say the {\em order-$2$ Serre relation for $(a(z_1), a(z_2), b(w))$} holds if
\begin{align}
\sum_{\sigma\in S_2}\!
  \Big(a(z_{\sigma(1)})a(z_{\sigma(2)})b(w)
  -2a(z_{\sigma(1)})b(w)a(z_{\sigma(2)})
    +b(w)a(z_{\sigma(1)})a(z_{\sigma(2)})\Big)\!=0.
\end{align}
\end{de}

The following were obtained in  \cite{KL1}:

\begin{prop}\label{Serre-relation-equivalence-va}
Let $V$ be an $\hbar$-adic nonlocal vertex algebra with $a,b\in V$ such that
\begin{align}
&  (z-w-2\hbar)Y(a,z)Y(a,w)=(z-w+2\hbar)Y(a,w)Y(a,z),\label{thm-a-a-relation-prop}\\
&  (z-w+\hbar)Y(a,z)Y(b,w)=(z-w-\hbar)Y(b,w)Y(a,z).\label{thm-a-b-relation-prop}
\end{align}
Then the order-$2$ Serre relation for $(Y(a,z_1), Y(a,z_2), Y(b,w))$ holds if and only if
\begin{align}\label{singular-Y-a-b}
{\rm Sing}_z Y(a,z)a_0b=0.
\end{align}
\end{prop}

\begin{prop}\label{Serre-relation-abstract}
Let $V$ be an $\hbar$-adic nonlocal vertex algebra and let $a,b\in V$. If
(\ref{thm-a-a-relation-prop}), (\ref{thm-a-b-relation-prop}, and
the order-$2$ Serre relation for $(Y(a,z_1), Y(a,z_2), Y(b,w))$ hold,
then for any $V$-module $(W,Y_W)$,
\begin{align}
&  (z-w-2\hbar)Y_W(a,z)Y_W(a,w)=(z-w+2\hbar)Y_W(a,w)Y_W(a,z),\label{thm-a-a-relation-module}\\
&  (z-w+\hbar)Y_W(a,z)Y_W(b,w)=(z-w-\hbar)Y_W(b,w)Y_W(a,z),\label{thm-a-b-relation-module}
\end{align}
and the order-$2$ Serre relation for $(Y_W(a,z_1), Y_W(a,z_2), Y_W(b,w))$ holds.
The converse is also true for any  faithful $V$-module $(W,Y_W)$.
\end{prop}

Set
\begin{align}\label{L-def}
L(x)=\sum_{n\ge 1}\frac{1}{n!}x^{n-1}=\frac{e^{x}-1}{x}\in \C[[x]].
\end{align}

\begin{prop}\label{Y-W-E(a,z)}
Let $V$ be an $\hbar$-adic nonlocal vertex algebra, $(W,Y_W)$ a $V$-module.
Assume $a\in V$ such that
\begin{eqnarray}\label{a-conditions}
[Y^{\pm}(a,x_1),Y^{\pm}(a,x_2)]=0, \quad
[Y^{-}(a,x_1),Y^{+}(a,x_2)]=\gamma(x_1-x_2),
\end{eqnarray}
where $Y^{+}(a,x)$ and $Y^{-}(a,x)$ denote the regular and singular parts of $Y(a,x)$, respectively,
and where $\gamma(x)\in x^{-1} \C[x^{-1}][[\hbar]]$.  Let $z\in \hbar \C[[\hbar]]$ and set
$$E^{\pm }(a,z)=\exp\! \(\sum_{n\in \pm \Z_{+}}\frac{1}{n}a_{n}z^{-n} \)$$
on $V$ and $W$, where $Y(a,x)=\sum_{n\in \Z}a_nx^{-n-1}$. Then
\begin{align}
&Y_W(E^{-}(-a,z){\bf 1},x)=\exp\! \(zL(z\partial_x)Y_W^{+}(a,x)\)\exp\! \(zL(z\partial_x)Y_W^{-}(a,x)\)\\
&\quad = \!\(1+\frac{z}{x}\)^{a_0}E^{-}(-a,x+z)E^{-}(a,x)E^{+}(-a,x+z)E^{+}(a,x). \nonumber
\end{align}
\end{prop}

\section{Algebra $\wh{\mathcal{DY}}(A)$ and $\hbar$-adic weak quantum vertex algebras $\V_A(\ell)$}\label{sec:DY}

In this section, we  recall from \cite{KL1} the $\C[[\hbar]]$-algebra $\wh{\mathcal{DY}}(A)$ associated to a simply-laced GCM $A$
and the universal vacuum $\wh{\mathcal{DY}}(A)$-module $\V_A(\ell)$ of level $\ell$.

Let $A=[a_{i,j}]_{i,j\in I}$ be a simply-laced GCM.
Then $A$ is symmetric and
$$a_{i,j}\in \{-1, 0, 2\}\quad  \text{ for }i,j\in I.$$

\begin{de}
Denote by $\wh{\mathcal{DY}}(A)$ the associative $\C[[\hbar]]$-algebra,
generated by
\begin{align}\label{eq:DY-gen-set}
\set{H_{i,n},\, X_{i,n}^\pm}{i\in I,\,n\in\Z}\cup \{\kappa\},
\end{align}
subject to a set of relations written in terms of generating functions in $z$:
\begin{align}
  &H_i^{+}(z)= 1+2\hbar \sum_{n\ge 0}H_{i,n}z^{-n-1},\quad
   H_i^{-}(z)=
    1+2\hbar \sum_{n\ge 1}H_{i,-n}z^{n-1},\\
 &\hspace{2.5cm}X_i^\pm(z)=\sum_{n\in\Z}X_{i,n}^\pm z^{-n-1}.
\end{align}
In addition to that $\kappa$ and $H_{i,0}$ ($i\in I$) are central, the relations are with $i,j\in I$:
\begin{align*}
  &\te{(DY1)}\quad [H_i^\pm(z),H_j^\pm(w)]=0,\\
  &\te{(DY2)}\quad H_i^+(z)H_j^-(w)=H_j^-(w)H_i^+(z)
  \frac{(z-w-a_{i,j}\hbar-\kappa\hbar)(z-w+a_{i,j}\hbar+\kappa\hbar)}
    {(z-w+a_{i,j}\hbar-\kappa\hbar)(z-w-a_{i,j}\hbar+\kappa\hbar)},\\
  &\te{(DY3)}\quad H_i^+(z)X_j^\pm(w)=X_j^\pm(w)H_i^+(z)
    \(\frac{z-w+a_{i,j}\hbar\pm\half \kappa\hbar}{z-w-a_{i,j}\hbar\pm\half \kappa\hbar}\)^{\pm 1},\\
  &\te{(DY4)}\quad H_i^-(z)X_j^\pm(w)=X_j^\pm(w)H_i^-(z)
    \(\frac{w-z-a_{j,i}\hbar\pm\half k\hbar}{w-z+a_{j,i}\hbar\pm\half \kappa\hbar}\)^{\pm 1},\\
  &\te{(DY5)}\quad [X_i^+(z),X_j^-(w)]
    =\delta_{i,j}\frac{1}{2\hbar}
    \Bigg(\!
        H_i^+(w+\half \kappa\hbar)z\inv\delta\!\(\frac{w+\kappa\hbar}{z}\)\\
  &\qquad\qquad\qquad\qquad
    -   H_i^-(w-\half \kappa\hbar)z\inv\delta\!\(\frac{w-\kappa\hbar}{z}\)\!
    \Bigg)\!,\\
  &\te{(DY6)}\quad (z-w\mp a_{i,j}\hbar)X_i^\pm(z)X_j^\pm(w)
    =(z-w\pm a_{i,j}\hbar)X_j^\pm(w)X_i^\pm(z),\\
  &\te{(DY7)}\quad X_i^\pm(z)X_j^\pm(w)
    =
    X_j^\pm(w)X_i^\pm(z)\quad \text{if }a_{i,j}= 0,\\
  &\te{(DY8)}\quad\sum_{\sigma\in S_2}
  \Big(X_i^\pm(z_{\sigma(1)})X_i^\pm(z_{\sigma(2)})X_j^\pm(w)
  -2X_i^\pm(z_{\sigma(1)})X_j^\pm(w)X_i^\pm(z_{\sigma(2)})\\
  &\qquad\qquad
    +X_j^\pm(w)X_i^\pm(z_{\sigma(1)})X_i^\pm(z_{\sigma(2)})\Big)=0
    \quad \te{if }a_{i,j}=-1.
\end{align*}
\end{de}

A $\wh{\mathcal{DY}}(A)$-module $W$ is said to be \emph{restricted}  if
$W$ is a topologically free $\C[[\hbar]]$-module and if for every $i\in I,\ w\in W$,
\begin{align}
\lim_{n\rightarrow \infty}H_{i,n}w=0=\lim_{n\rightarrow \infty}X_{i,n}^{\pm}w.
\end{align}

\begin{de}\label{def-A}
Denote by $\widetilde{\mathcal{A}}$ the associative algebra over $\C$, generated by elements
$$\tilde{\kappa},\ \tilde{H}_{i,n},\ \tilde{X}_{i,n}^{\pm}\ \  (\text{where }i\in I,\ n\in \Z),$$
subject to the condition that $\tilde{\kappa}$ and $\tilde{H}_{i,0}$ for $i\in I$ are central.
\end{de}

Consider the $\C[[\hbar]]$-algebra $\widetilde{\mathcal{A}}[[\hbar]]$ and
formulate generating functions $\tilde{H}_i^{\pm}(z),\ \tilde{X}_i^{\pm}(z)$ for $i\in I$ in the same way
as for $H_i^{\pm}(z),\ X_i^{\pm}(z)$. Then $\tilde{H}_i^{+}(z)$ and $\tilde{H}_i^{-}(z)$ are invertible elements of
$\widetilde{\mathcal{A}}[[z^{-1}]][[\hbar]]$ and $\widetilde{\mathcal{A}}[[z]][[\hbar]]$, respectively.
A {\em restricted} $\widetilde{\mathcal{A}}[[\hbar]]$-module by definition is an  $\widetilde{\mathcal{A}}[[\hbar]]$-module $W$
which is a topologically free  $\C[[\hbar]]$-module such that
\begin{align}
\lim_{n\rightarrow \infty}\tilde{H}_{i,n}w=0=\lim_{n\rightarrow \infty}\tilde{X}_{i,n}^{\pm}w
\quad \text{ for }i\in I,\ w\in W.
\end{align}

Set $q=e^{\hbar}\in \C[[\hbar]]$, and introduce formal power series
\begin{align}
&G(x)=\frac{q^{x}-q^{-x}}{x}:=2\hbar\(\!1+\sum_{n\ge 1}\frac{\hbar^{2n}}{(2n+1)!}x^{2n}\!\)\in \hbar\C[[\hbar,x]],\\
&F(x)= \frac{1}{2\hbar}\(\!1+\sum_{n\ge 1}\frac{\hbar^{2n}}{(2n+1)!}x^{2n}\!\)^{-1}\in \hbar^{-1}\C[[x,\hbar]].
\end{align}

\begin{de}\label{new-operators}
Let $W$ be a restricted $\widetilde{\mathcal{A}}[[\hbar]]$-module. For $i\in I$, set
\begin{align}
  h_{i,\Y}^\pm(z) =\pm F(\partial_z) \log \tilde{H}_i^\pm\big(z\pm  \frac{1}{2}\tilde{\kappa}\hbar \big)\in \mathcal{A}[[z^{\mp 1}]][[\hbar]].
    \label{eq:def-h-x-0}
  \end{align}
 Furthermore, set
  \begin{align}
  &h_{i,\Y}(z)=h_{i,\Y}^+(z)+h_{i,\Y}^-(z),\\
  &x_{i,\Y}^+(z)=\tilde{X}_i^+(z),\label{eq:def-h-x-1}\\
  &x_{i,\Y}^-(z)=\tilde{X}_i^-(z-\tilde{\kappa}\hbar)\tilde{H}_i^+\big(z-\frac{1}{2} \tilde{\kappa}\hbar\big)\inv.\label{eq:def-h-x-2}
\end{align}
\end{de}

The following was proved in \cite{KL1}:

\begin{thm}\label{thm-main-DY}
Let $W$ be a restricted $\widetilde{\mathcal{A}}[[\hbar]]$-module.
For $i\in I$, set
\begin{align}
C_i(x)&=\tilde{H}^{-}_i(x-\frac{3}{2}\tilde{\kappa}\hbar)\tilde{H}_i^+(x-\frac{1}{2}\tilde{\kappa}\hbar)^{-1} \\
&=\exp\!\(-G(\partial_x)q^{-\tilde{\kappa}\partial_x}h_{i,\Y}^{-}(x)\)
\exp\!\(-G(\partial_x)q^{-\tilde{\kappa}\partial_x}h_{i,\Y}^{+}(x)\)\!,\nonumber
\end{align}
an element of $\E_{\hbar}(W)$.
Then $W$ is a restricted $\wh{\mathcal{DY}}(A)$-module with
$$\kappa=\tilde{\kappa}, \ H_{i,n}=\tilde{H}_{i,n},\ X_{i,n}^{\pm}=\tilde{X}_{i,n}^{\pm} \quad \text{ for }i\in I$$
if and only if the following relations hold for $i,j\in I$:
\begin{align*}
&[h_{i,\Y}(z),h_{j,\Y}(w)]
    =[a_{i,j}]_{q^{\partial_{w}}}
        [\kappa]_{q^{\partial_{w}}}\( (z-w+\hbar \kappa)^{-2}-(w-z+\hbar \kappa)^{-2}\)\!,\\
&[h_{i,\Y}(z),x_{j,\Y}^\pm(w)]
=\pm x_{j,\Y}^\pm(w)[a_{i,j}]_{q^{\partial_{w}}}\((z-w+\hbar \kappa)^{-1}+(w-z+\hbar \kappa)^{-1}\)\!,\\
&(z-w-a_{i,j}\hbar)x_{i,\Y}^\pm(z)x_{j,\Y}^\pm(w)=(z-w+a_{i,j}\hbar)x_{j,\Y}^\pm(w)x_{i,\Y}^\pm(z),\\
&x_{i,\Y}^\pm(z)x_{j,\Y}^\pm(w)=
x_{j,\Y}^\pm(w)x_{i,\Y}^\pm(z)\ \ \text{ if } a_{i,j}= 0,\\
&x_{i,\Y}^+(z)x_{j,\Y}^-(w)  - \( \frac{w-z+a_{i,j}\hbar}{w-z-a_{i,j}\hbar}\)x_{j,\Y}^-(w)x_{i,\Y}^+(z)\nonumber\\
&\quad  =\delta_{i,j}\frac{1}{2\hbar}\!
\(\!z\inv\delta\!\(\frac{w}{z}\)-C_i(w)  z\inv\delta\!\(\frac{w-2\kappa\hbar}{z}\)\!\)\!,
\end{align*}
and the order-$2$ Serre relations for $(x_{i,\Y}^\pm(z_1), x_{i,\Y}^\pm(z_2),x_{j,\Y}^\pm(w))$ hold when $a_{i,j}=-1$.
\end{thm}

With Theorem \ref{thm-main-DY},
from now on we use $x_{i,\Y}^\pm(z), \ h_{i,\Y}(z)$  $(i\in I)$ as the generating functions for $\wh{\mathcal{DY}}(A)$.
For $i\in I$, write
\begin{align}
x_{i,\Y}^{\pm}(z)=\sum_{n\in \Z}x_{i,\Y}^{\pm}(n)z^{-n-1},\quad h_{i,\Y}(z)=\sum_{n\in \Z}h_{i,\Y}(n)z^{-n-1}.
\end{align}
Note that for $\ell\in \C$,
\begin{align}\label{Ci-G-L}
&\exp\!\(-G(\partial_x)q^{-\ell\partial_x}h_{i,\Y}^{-}(x)\)
\exp\!\(-G(\partial_x)q^{-\ell\partial_x}h_{i,\Y}^{+}(x)\)\\
=\ & q^{(1-\ell)\partial_x}\!\(\exp \(-2\hbar L(-2\hbar \partial_x)h_{i,\Y}^{-}(x)\)\!\exp \(-2\hbar L(-2\hbar \partial_x)h_{i,\Y}^{+}(x)\)\)\!.\nonumber
\end{align}

\begin{de}
{\em A vacuum $\wh{\mathcal{DY}}(A)$-module} of level $\ell\in \C$ is a pair $(W,w_0)$, where $W$ is a
restricted $\wh{\mathcal{DY}}(A)$-module of level $\ell$ and $w_0$ is a (vacuum) vector in $W$ such that
 $W=(\wh{\mathcal{DY}}(A)w_0)[[\hbar]]'$  and
\begin{align}
x_{i,\Y}^{\pm}(n)w_0=0=h_{i,\Y}(n)w_0\quad \text{ for }i\in I,\ n\ge 0.
\end{align}
\end{de}

A vacuum $\wh{\mathcal{DY}}(A)$-module $(V,{\bf 1})$ of level $\ell$ is said to be {\em universal} if
for any vacuum $\wh{\mathcal{DY}}(A)$-module $(W,w_0)$ of level $\ell$, there exists a unique
$\wh{\mathcal{DY}}(A)$-module homomorphism $\psi: V\rightarrow W$ such that $\psi({\bf 1})=w_0$.

Next, we recall the construction of  the universal vacuum $\wh{\mathcal{DY}}(A)$-module $\V_A(\ell)$ of level $\ell$.
First, let $\widetilde{\mathcal{U}}$ be the associative algebra with identity over $\C$, generated by set
$$\{\tilde{\kappa}\}\cup \{\tilde{e}_i(m), \tilde{f}_i(m), \tilde{h}_i(m)\ |\ i\in I,\ m\in \Z\}.$$
For $X\in \{ \tilde{e}_i,\tilde{f}_i, \tilde{h}_i\ |\ i\in I\}$, form a generating function
\begin{align}
X(z)=\sum_{m\in \Z}X(m)z^{-m-1}.
\end{align}

Second, consider the $\C[[\hbar]]$-algebra $\widetilde{\mathcal{U}}[[\hbar]]$, and
define $\widetilde{J}_{A,\ell}$ to be the two-sided ideal,
generated by element $\tilde{\kappa}-\ell$ and by the following relations for $i,j\in I$:
\begin{align}
&\tilde{h}_i(z)\tilde{h}_j(w)-\tilde{h}_j(w)\tilde{h}_i(z)
=[a_{i,j}]_{q^{\partial_w}}[\ell]_{q^{\partial_w}}((z-w+\ell\hbar)^{-2}-(w-z+\ell\hbar)^{-2}),\\
&\tilde{h}_i(z)\tilde{e}_j(w)-\tilde{e}_j(w)\tilde{h}_i(z)=\tilde{e}_j(w)[a_{i,j}]_{q^{\partial_w}}((z-w+\ell\hbar)^{-1}-(w-z+\ell\hbar)^{-1}),\\
&\tilde{h}_i(z)\tilde{f}_j(w)-\tilde{f}_j(w)\tilde{h}_i(z)=-\tilde{f}_j(w)[a_{i,j}]_{q^{\partial_w}}((z-w+\ell\hbar)^{-1}-(w-z+\ell\hbar)^{-1}),\\
&(z-w-a_{i,j}\hbar)\tilde{e}_i(z)\tilde{e}_j(w)=(z-w+a_{i,j}\hbar)\tilde{e}_j(w)\tilde{e}_i(z),\label{e-i-e-j}\\
&(z-w-a_{i,j}\hbar)\tilde{f}_i(z)\tilde{f}_j(w)=(z-w+a_{i,j}\hbar)\tilde{f}_j(w)\tilde{f}_i(z),\label{f-i-f-j}\\
&\label{i-not-j}
(z-w+a_{i,j}\hbar) \tilde{e}_i(z)\tilde{f}_j(w)=(z-w-a_{i,j}\hbar) \tilde{f}_j(w)\tilde{e}_i(z)\quad \text{ if }i\ne j,\\
&\label{i=j}
(z-w)(z-w+2\ell\hbar)\!\((z-w+2\hbar) \tilde{e}_i(z)\tilde{f}_i(w)-(z-w-2\hbar) \tilde{f}_i(w)\tilde{e}_i(z)\)\!=0,
\end{align}
and the order-$2$ Serre relations for $(\tilde{e}_i(z_1),  \tilde{e}_i(z_1), \tilde{e}_j(w))$ and $(\tilde{f}_i(z_1),  \tilde{f}_i(z_1), \tilde{f}_j(w))$
for $i,j\in I$ with $a_{i,j}=-1$.

By Lemma \ref{basic-facts-top-algebra}, $\overline{[\widetilde{J}_{A,\ell}]}$ is a two-sided ideal of $\widetilde{\mathcal{U}}[[\hbar]]$
and a topologically free $\C[[\hbar]]$-submodule.  Then set
\begin{align}
\mathcal{U}_{A,\ell}=\widetilde{\mathcal{U}}[[\hbar]]/ \overline{[\widetilde{J}_{A,\ell}]},
\end{align}
which is an associative $\C[[\hbar]]$-algebra and
a topologically free $\C[[\hbar]]$-module by Lemma \ref{strong-submodule}. For $i\in I,\ m\in \Z$, set
\begin{align}
\bar{e}_i(m)=\tilde{e}_i(m)+\overline{[\widetilde{J}_{A,\ell}]},\  \bar{f}_i(m)=\tilde{f}_i(m)+\overline{[\widetilde{J}_{A,\ell}]},\
\bar{h}_i(m)=\tilde{h}_i(m)+\overline{[\widetilde{J}_{A,\ell}]}\in \mathcal{U}_{A,\ell}.
\end{align}
Then define generating functions $\bar{e}_i(z),\ \bar{f}_i(z),\ \bar{h}_i(z)$ for $i\in I$ correspondingly.

Third, let $\mathcal{J}_{A,\ell}$ be the left ideal of $\mathcal{U}_{A,\ell}$, generated by the subset
$$\{ \bar{e}_i(n),\bar{f}_i(n), \bar{h}_i(n)\ |\ i\in I,\ n\ge 0\}.$$
Then $\overline{[\mathcal{J}_{A,\ell}]}$ is a topologically free $\C[[\hbar]]$-submodule and a left ideal of $\mathcal{U}_{A,\ell}$.
Define
\begin{align}
\mathcal{V}_{A,\ell}=\mathcal{U}_{A,\ell}/\overline{[\mathcal{J}_{A,\ell}]},
\end{align}
which is a (left) $\mathcal{U}_{A,\ell}$-module and a  topologically free $\C[[\hbar]]$-module.  Set
$${\bf 1}=1+\overline{[\mathcal{J}_{A,\ell}]}\in \mathcal{V}_{A,\ell},$$
and furthermore set
\begin{align}
\hat{e}_i=\bar{e}_i(-1){\bf 1},\ \  \hat{f}_i=\bar{f}_i(-1){\bf 1},\ \ \hat{h}_i=\bar{h}_i(-1){\bf 1}\in \mathcal{V}_{A,\ell}
\quad \text{for }i\in I.
\end{align}

We have (see \cite{KL1}):

\begin{prop}\label{pre-universal-qva}
There exists an $\hbar$-adic weak quantum vertex algebra structure on $\mathcal{V}_{A,\ell}$, which is uniquely determined by
the condition that ${\bf 1}$ is the vacuum vector and
\begin{align}
Y(\hat{e}_i,z)=\bar{e}_i(z),\ \ Y(\hat{f}_i,z)=\bar{f}_i(z),\ \ Y(\hat{h}_i,z)=\bar{h}_i(z)\quad \text{for }i\in I.
\end{align}
\end{prop}

For $i\in I$, set
\begin{align}
C_i=e^{(1-\ell)\hbar \D}E^{-}(-\hat{h}_i,-2\hbar){\bf 1}\in \mathcal{V}_{A,\ell},
\end{align}
where $\D$ is the canonical $D$-operator of $\mathcal{V}_{A,\ell}$ and
$$ E^{-}(-\hat{h}_i,-2\hbar)=\exp\!\(\sum_{n\in \Z_{+}}\frac{1}{n}\hat{h}_i(-n)(-2\hbar)^n\right)\!.$$
Note that $C_i={\bf 1}-2\hbar \hat{h}_i+O(\hbar^2)$ (as $\D {\bf 1}=0$).

Denote by $K$ the subset of $\mathcal{V}_{A,\ell}$, consisting of vectors
 \begin{align}
 (\hat{e}_i)_0(\hat{f}_j)-\frac{1}{2\hbar}\delta_{i,j}({\bf 1}-C_i),\ \
 (\hat{e}_i)_n(\hat{f}_j)-\delta_{i,j}\ell (-2\ell\hbar)^{n-1}C_i\ \ \text{ for }i,j\in I,\ n\ge 1
 \end{align}
 and vectors
 \begin{align}
 (\hat{e}_i)_n(\hat{e}_j),\ (\hat{f}_i)_n(\hat{f}_j)\quad \text{ for } i,j\in I\ \text{with } a_{i,j}= 0\ \text{and for }n\ge 0.
 \end{align}

 \begin{de}
 Let $(K)$ denote the strong ideal of $\mathcal{V}_{A,\ell}$, generated by $K$.
 Set
 \begin{align}
 {\mathcal{V}}_A(\ell)=\mathcal{V}_{A,\ell}(\ell)/(K),
 \end{align}
which is a topologically free $\C[[\hbar]]$-module and
an $\hbar$-adic weak quantum vertex algebra.
\end{de}

Set
\begin{align}
\check{e}_i=\hat{e}_i+(K),\ \check{f}_i=\hat{f}_i+(K),\ \check{h}_i=\hat{h}_i+(K)\in {\mathcal{V}}_A(\ell)
\end{align}
for $i\in I$, and furthermore set
\begin{align}
\mathfrak{c}_i=e^{(1-\ell)\hbar \D}E^{-}(-\check{h}_i,-2\hbar){\bf 1}\in {\mathcal{V}}_A(\ell),
\end{align}
where $\D$ is the $D$-operator of ${\mathcal{V}}_A(\ell)$.
For $i,j\in I$, we have
   \begin{align}\label{key-relation}
 &Y(\check{e}_i,z)Y(\check{f}_j,w) -\! \(\frac{w-z+a_{i,j}\hbar}{w-z-a_{i,j}\hbar}\)\!Y(\check{f}_j,w)Y(\check{e}_i,z)\\
 =\ &\delta_{i,j}\frac{1}{2\hbar}\!\left(\!z^{-1}\delta\!\left(\frac{w}{z}\right)-Y(\mathfrak{c}_i,w)z^{-1}\delta\!\left(\frac{w-2\ell\hbar}{z}\right)\!\)\!.\nonumber
  \end{align}
Furthermore, using Proposition \ref{Y-W-E(a,z)} and relation (\ref{Ci-G-L}) we obtain
 \begin{align}\label{explicit-expression}
 &Y(\mathfrak{c}_i,x) \\
 =\ &e^{(1-\ell)\hbar \partial_x}Y\!\(E^{-}(-\check{h}_i,-2\hbar){\bf 1},x\)\nonumber\\
= \ &e^{(1-\ell)\hbar \partial_x}
\exp\! \(-2\hbar L(-2\hbar\partial_x)Y^{+}(\check{h}_i,x)\)\exp\! \(-2\hbar L(-2\hbar\partial_x)Y^{-}(\check{h}_i,x)\)\nonumber\\
 =\ &\exp\!\(-G(\partial_x)q^{-\ell\partial_x}Y^{+}(\check{h}_i,x)\)\exp\!\(-G(\partial_x)q^{-\ell\partial_x}Y^{-}(\check{h}_i,x)\)\!.\nonumber
  \end{align}
For $i,j\in I$ with $a_{i,j}= 0$, we have $(\check{e}_i)_n(\check{e}_j)=0=(\check{f}_i)_n(\check{f}_j)$ for $n\ge 0$, and
\begin{align}
&Y(\check{e}_i,z)Y(\check{e}_j,w)= 
Y(\check{e}_j,w)Y(\check{e}_i,z),\\
&Y(\check{f}_i,z)Y(\check{f}_j,w)= 
Y(\check{f}_j,w)Y(\check{f}_i,z).
\end{align}

The following  were proved in \cite{KL1}:

\begin{prop}
 Let $\ell\in \C$. Then the $\hbar$-adic weak quantum vertex algebra $\V_A(\ell)$
 is a restricted $\wh{\mathcal{DY}}(A)$-module of level $\ell$ with
 \begin{align}
 x_{i,\Y}^{+}(z)=Y(\check{e}_i,z),\ \ x_{i,\Y}^{-}(z)=Y(\check{f}_i,z),\ \ h_{i,\Y}(z)=Y(\check{h}_i,z)\ \ \text{ for }i\in I.
 \end{align}
 Furthermore, $({\mathcal{V}}_A(\ell), {\bf 1})$ is a universal vacuum $\wh{\mathcal{DY}}(A)$-module
 of level $\ell$.
\end{prop}

\begin{thm}\label{hqva-module-main}
Let $\ell\in \C$. For any restricted $\wh{\mathcal{DY}}(A)$-module $W$ of level $\ell$, there exists a
$\V_A(\ell)$-module structure $Y_W(\cdot,z)$ which is uniquely determined by
 \begin{align}
 Y_W(\check{e}_{i},z)=x_{i,\Y}^{\pm}(z),\ \ Y_W(\check{f}_{i},z)=x_{i,\Y}^{\pm}(z),\ \ Y_W(\check{h}_{i},z)=h_{i,\Y}(z)\ \ \text{ for }i\in I.
 \end{align}
On the other hand, for any $\V_A(\ell)$-module $(W,Y_W)$,
there exists a restricted
$\wh{\mathcal{DY}}(A)$-module structure of level $\ell$ on $W$ such that
 \begin{align}
x_{i,\Y}^{+}(z)=Y_W(\check{e}_{i},z),\ \ x_{i,\Y}^{-}(z)=Y_W(\check{f}_{i},z),\ \ h_{i,\Y}(z)=Y_W(\check{h}_i,z)\ \ \text{ for }i\in I.
 \end{align}
\end{thm}

\section{$\hbar$-adic quantum vertex algebra $V_L[[\hbar]]^{\eta}$}\label{sec:deform-L-Q}

This section is the core of this paper.
Let $A$ be a simply-laced GCM and let $L$ be the root lattice of the Kac-Moody algebra $\g(A)$.
As the main results, by applying a general result of \cite{JKLT-22} we get a particular $\hbar$-adic quantum vertex algebra
$V_L[[\hbar]]^{\eta}$, and we show that
every $V_L[[\hbar]]^{\eta}$-module is naturally a $\widehat{\mathcal{DY}}(A)$-module of level one
 and $V_L[[\hbar]]^{\eta}$ is a homomorphism image of
the $\hbar$-adic weak quantum vertex algebra $\V_A(1)$. In case that $A$ is of finite $ADE$ type,
we obtain a precise description of $V_L[[\hbar]]^{\eta}$ as a quotient of the $\hbar$-adic weak quantum vertex algebra $\V_A(1)$.

Recall that $\g(A)$ is equipped with the standard Cartan subalgebra $\h$ with the sets of simple roots and coroots
$\set{\al_i}{i\in I}\subset\h^\ast$ and $\set{h_i}{i\in I}\subset\h,$ respectively, where
\begin{align}
  \al_j(h_i)=a_{i,j}\quad \mbox{ for }i,j\in I.
\end{align}
Set
\begin{align}
L=\sum_{i\in I}\Z\alpha_i\subset \h^*,
\end{align}
the root lattice of $\g(A)$.
Fix a non-degenerate symmetric invariant bilinear form $\<\cdot,\cdot\>$
 on $\g(A)$ such that $\<h_i,h_j\>=a_{i,j}$ for $i,j\in I$.
 Note that $\<\cdot,\cdot\>$ is also non-degenerate on $\h$.
Identify $\h^\ast$ with $\h$ via  $\<\cdot,\cdot\>$. Then $\<\alpha_i,\alpha_j\>=a_{i,j}$ for $i,j\in I$ and
\begin{align*}
  \<\al,\al\>\in 2\Z\quad  \te{for }\al\in L.
\end{align*}

Form an (affine) Lie algebra $\wh\h:=\h\ot\C[t,t\inv]\oplus\C {\bf k}$,
where  $[{\bf k},\wh \h]=0$ and
\begin{align}
  [h(m),h'(n)]=m\delta_{m+n,0}\<h,h'\>{\bf k}
\end{align}
for $h,h'\in\h$, $m,n\in\Z$ with $h(m):=h\ot t^m$. Identify $\h$ with $\h\ot t^0$.
Set
\begin{align}
  \wh\h^\pm=\h\ot t^{\pm 1}\C[t^{\pm 1}],
\end{align}
which are abelian Lie subalgebras of $\wh\h$.  Furthermore, set
\begin{align*}
  \wh\h'=\wh\h^++\wh\h^-+\C {\bf k},
\end{align*}
which is a Heisenberg algebra.
Then $\wh\h=\wh\h'\oplus\h$ is a direct sum of Lie algebras.
The symmetric algebra $S(\wh\h^-)$ affords a canonical level-one irreducible representation of $\wh\h'$.
Let $\h$ act trivially to make $S(\wh\h^-)$ an $\wh\h$-module (of level one).

Let $\varepsilon:L\times L\to \C^\times$ be a $2$-cocycle such that
\begin{align*}
 \varepsilon(\al,\be)\varepsilon(\be,\al)\inv=(-1)^{\<\al,\be\>},\quad
  \varepsilon(\al,0)=1=\varepsilon(0,\al)\   \quad \te{for } \al,\be\in L.
\end{align*}
The $\varepsilon$-twisted group algebra $\C_\varepsilon[L]$ of $L$ has a designated basis $\set{e_\al}{\al\in L}$ with
$$e_\al\cdot e_\be=\varepsilon(\al,\be)e_{\al+\be}\quad \text{ for }\al,\be\in L.$$
Make $\C_\varepsilon[L]$ an $\wh\h$-module by letting $\wh\h'$ act trivially and letting $\h$ act by
\begin{align}
  h\cdot e_\al=\al(h)e_\al\quad \te{for }h\in\h,\  \al\in L.
\end{align}
Set
\begin{align}
V_{(\h,L)}=S(\wh\h^-)\ot \C_\varepsilon[L],
\end{align}
equipped with the tensor product $\wh\h$-module structure.
View $\h$ and $\C_\varepsilon[L]$ as subspaces of $V_{(\h,L)}$ via the maps
$h\mapsto h(-1)\ot 1$ ($h\in\h$) and $e_\al\mapsto 1\ot e_\al$ ($\al\in L$).

Set
\begin{align}
  \vac=e_0\in\C_\varepsilon[L]\subset V_{(\h,L)}.
\end{align}
For $\al\in L$, a linear operator $z^\al:\ V_{(\h,L)}\to V_{(\h,L)}[z,z\inv]$ is defined by
\begin{align}
  z^\al(u\otimes e_\be)=z^{\<\al,\be\>}(u\otimes e_\be)\quad \te{for }u\in S(\wh\h^-),\ \be\in L.
\end{align}
On the other hand, for $a\in \h$ set
\begin{align}\label{E-pm-h,z}
E^{\pm}(a,z)= \exp\!\(\sum_{\pm n\in \Z_{+}}\frac{a(n)}{n}z^{-n}\)\!.
\end{align}
Then (see \cite{Bor}, \cite{FLM}, cf. \cite{LX}, \cite{LW}),
there exists a vertex algebra structure on $V_{(\h,L)}$,
which is uniquely determined by the condition that ${\bf 1}$ is the vacuum vector and
\begin{align}
  &Y(h,z)=h(z):=\sum_{n\in\Z}h(n)z^{-n-1}\quad \te{ for }h\in \h,\\
  &Y(e_\al,z)=E^{-}(-\al,z)E^{+}(-\al,z)e_\al z^\al \quad \te{ for }\al\in L.
 \end{align}
Furthermore, $V_{(\h,L)}$ is a simple conformal vertex algebra
(which satisfies all the conditions for a vertex operator algebra except the two grading restrictions as in \cite{FLM}).

By the same arguments of \cite{JKLT-23} (cf.  \cite{JKLT-22}, \cite{LL}) we have:

\begin{prop}\label{Universal-property}
Let $V$ be any nonlocal vertex algebra with a linear map
$\psi:\h\oplus\C_{\varepsilon}[L]\rightarrow V$ such that $\psi(e_0)={\bf 1}$. Set
\begin{align*}
  h[z]= Y(\psi(h),z),\quad e_\al[z]= Y(\psi(e_\al),z)\quad
  \te{for }h\in\h,\  \al\in L.
\end{align*}
Assume that all the following relations hold for $h, h'\in \h,\  \al,\be\in L$:
\begin{align}
    & \   \   \    \left[h[x],
    h'[z]\right]=\<h,h'\>\frac{\partial}{\partial z}x\inverse\delta\!\(\frac{z}{x}\)\!,\\
  & \   \  \    \left[h[x],e_\al[z]\right]=\<\al,h\>
    e_\al[z]x\inverse\delta\!\(\frac{z}{x}\)\!,\\
  & \   \   \   \frac{d}{dz} e_\al[z]=\al[z]^{+} e_\al[z]+e_\al[z]\al[z]^{-},
\end{align}
where $\al[z]^{+}=\sum_{n<0}\al[n]z^{-n-1},\  \al[z]^{-}=\sum_{n\ge 0}\al[n]z^{-n-1}$, and
\begin{align*}
(x-z)^{-\<\al,\be\>-1}e_{\alpha}[x]e_{\be}[z]
-(-z+x)^{-\<\al,\be\>-1}e_{\be}[z]e_{\alpha}[x]
=  \varepsilon(\al,\be)e_{\al+\be}[z]x^{-1}\delta\!\left(\frac{z}{x}\right)\!.
\end{align*}
Then $\psi$ extends uniquely to a nonlocal vertex algebra morphism from $V_{(\h,L)}$ to $V$. Furthermore, if
$V$ is generated by $\psi(\h+\C_{\varepsilon}[L])$, then $\psi$ is an isomorphism.
\end{prop}

Set
\begin{align}
\h_0=\sum_{i\in I}\C \al_i\subset \h\cap [\g(A),\g(A)].
\end{align}
(Note that $L$ is a free abelian group and $\h_0=\C\ot_{\Z}L$.)
Furthermore, set
\begin{align}
  V_L=S(\wh{\h_0}^-)\ot\C_\varepsilon[L]\subset V_{(\h,L)},
\end{align}
where $\wh{\h_0}^-=\h_0\ot t\inv\C[t\inv]\subset \wh{\h}^-$.
It can be readily seen that $V_L$ is a vertex subalgebra of $V_{(\h,L)}$,
which is generated by the subspace $\h_0+\C_\varepsilon[L]$.

For $a\in \h,\ f(z)\in \C((z))[[\hbar]]$, follow \cite{EK-qva} to define
\begin{align}\label{eq:formula-Phi-pre}
\Phi(a,f)(x)=\Res_z f(x-z)Y(a,z)=\sum_{n\ge 0}\frac{(-1)^{n}}{n!}f^{(n)}(x)a_{n},
\end{align}
 where $f^{(n)}(x)=(\frac{d}{dx})^nf(x)$. A basic property of $\Phi(a,f)(x)$ is
 \begin{align}\label{Phi-pseudo-end}
 \Phi(a,f)(x) Y(v,z)=Y(v,z)\Phi(a,f)(x)+Y\!\(\Phi(a,f)(x-z)v,z\)
 \end{align}
for $v\in V_{(\h,L)}$.  On the other hand, we have
\begin{align}\label{formula-Phi}
&\Phi(a,f)(x)b=-\<a,b\>f'(x){\bf 1}\   \   \   \mbox{ for }b\in \h,\\
&\Phi(a,f)(x)e_{\be}=\<\be,a\>f(x)e_{\be}\   \   \   \mbox{ for }\be\in L.
\end{align}
Define $\Phi(G(x))$ for any $G(x)\in \h\otimes \C((x))[[\hbar]]$ in the obvious way.
Note that
\begin{align}\label{Phi-G(z)-V-L}
  \Phi(G(x))V_L\subset V_L((x))[[\hbar]]
\end{align}
as $a_nV_L\subset V_L$ for any $a\in \h,\ n\ge 0$.
On the other hand, extend the bilinear form $\<\cdot,\cdot\>$ on $\h$ to
a $\C((x))[[\hbar]]$-valued bilinear form on $\h\otimes \C((x))[[\hbar]]$. Then
\begin{align}\label{Phi-G(x)-e-beta}
\Phi(G(x))e_{\be}=\<\be,G(x)\>e_{\be}
\end{align}
for $G(x)\in \h\otimes \C((x))[[\hbar]],\ \be\in L$.

Denote by
${\rm Hom}(\h, \h\otimes \C((x))[[\hbar]])^{0}$ the space of linear maps
$\eta: \h\rightarrow \h\otimes \C((x))[[\hbar]]$ such that
$\eta_0(a,x):=\eta(a,x)|_{\hbar=0}\in \h\ot x\C[[x]]$ for $a\in \h$.

As a slight generalization of \cite[Theorem 5.9]{JKLT-22}, we have:

\begin{thm}\label{thm:qlva}
Let $\eta(\cdot,x)\in {\rm Hom}(\h, \h\otimes \C((x))[[\hbar]])^{0}$.
Then there exists an $\hbar$-adic quantum vertex algebra structure on $V_{(\h,L)}[[\hbar]]$,
 where the vertex operator map, denoted by $Y_L^\eta(\cdot,x)$, is uniquely determined by
\begin{eqnarray}
&&Y_L^{\eta}(a,x)=Y(a,x)+\Phi(\eta'(a,x)) \quad \text{ for }a\in \h,\label{Y-L-eta-1}\\
&&Y_L^{\eta}(e_{\alpha},x)=Y(e_\al,x)\exp(\Phi(\eta(\alpha,x)))\quad \text{ for }\alpha\in L.\label{Y-L-eta-2}
\end{eqnarray}
Denote this $\hbar$-adic quantum vertex algebra by $V_{(\h,L)}[[\hbar]]^\eta$.
\end{thm}

\begin{proof} In \cite[Theorem 5.9]{JKLT-22}, $L$ was assumed to be a non-degenerate even lattice
and the proof uses a vertex bialgebra $B_L$.
Here, we need a slight modification.  Set
\begin{align}
B_{(\h,L)}=S(\wh{\h}^-)\ot\C[L],
\end{align}
which is naturally a commutative associative algebra and a bialgebra.
It also admits a derivation $\partial$ which is uniquely determined by
\begin{align}
&\partial (e^{\alpha})=\alpha(-1)\otimes e^{\alpha}\quad \text{ for }\alpha\in L,\\
&\partial (u(-n))=nu(-n-1)\quad \text{ for }u\in \h,\ n\in \Z_{+}.
\end{align}
View $B_{(\h,L)}$ as a (commutative) vertex algebra with
$$Y(a,x)b=\(e^{x\partial}a\)b\quad \text{ for }a,b\in B_{(\h,L)}.$$
Then $B_{(\h,L)}$ becomes a vertex bialgebra.
As a commutative differential algebra, $B_{(\h,L)}$ has the following universal property:
For any commutative differential algebra $(K,\partial)$ and for any linear map
$\psi_0: \h\oplus \C[L]\rightarrow K$ such that $\psi_0(e^0)=1$ and
$$\partial \psi_0(e^{\alpha})=\psi_0(\alpha)\psi_0(e^{\alpha})\quad \text{ for }\alpha\in L,$$
there exists uniquely a homomorphism $\psi: B_{(\h,L)}\rightarrow K$, extending $\psi_0$.

Using the universal property of $V_{(\h,L)}$ and the arguments of \cite{JKLT-22},
we conclude that there exists a vertex-algebra homomorphism
$$\rho:\ V_{(\h,L)}\rightarrow  V_{(\h,L)}\otimes B_{(\h,L)},$$
which is uniquely determined by
\begin{align}
&\rho(e_{\alpha})=e_{\alpha}\otimes e^{\alpha}\quad \text{ for }\alpha\in L,\\
&\rho(a)=a\otimes e^0+{\bf 1}\otimes a(-1)\quad \text{ for }a\in \h.
\end{align}
Then $V_{(\h,L)}$ becomes a right $B_{(\h,L)}$-comodule vertex algebra.
Furthermore, $V_{(\h,L)}[[\hbar]]$ is a right $B_{(\h,L)}[[\hbar]]$-comodule $\hbar$-adic vertex algebra.
Then follow \cite{JKLT-22} to complete the proof.
\end{proof}

On the other hand, we also have:

\begin{prop}\label{qlva-module}
Let $\eta(\cdot,x)\in {\rm Hom}(\h, \h\otimes \C((x))[[\hbar]])^{0}$ and let $(W,Y_W)$ be any $V_{(\h,L)}$-module.
Then there exists a $V_{(\h,L)}[[\hbar]]^\eta$-module structure $Y_W^{\eta}(\cdot,x)$ on $W[[\hbar]]$, which is uniquely determined by
\begin{eqnarray}
&&Y_W^{\eta}(a,x)=Y_W(a,x)+\Phi(\eta'(a,x)) \quad \text{ for }a\in \h,\\
&&Y_W^{\eta}(e_{\alpha},x)=Y_W(e_\al,x)\exp(\Phi(\eta(\alpha,x)))\quad \text{ for }\alpha\in L.
\end{eqnarray}
\end{prop}

\begin{proof}  Just as in \cite{Li-smash}, $V_{(\h,L)}[[\hbar]]$ is naturally a $B_{(\h,L)}[[\hbar]]$-module $\hbar$-adic vertex algebra,
and then we have the smash product $\hbar$-adic nonlocal vertex algebra $V_{(\h,L)}[[\hbar]]\sharp B_{(\h,L)}[[\hbar]]$.
As in \cite{JKLT-22}, the map $\rho: V_{(\h,L)}[[\hbar]] \rightarrow  V_{(\h,L)}[[\hbar]]\wh\otimes B_{(\h,L)}[[\hbar]]$
from the proof of Theorem \ref{thm:qlva}
 is in fact an $\hbar$-adic nonlocal-vertex-algebra homomorphism
$$\rho:\ V_{(\h,L)}[[\hbar]]^{\eta}\rightarrow  V_{(\h,L)}[[\hbar]]\sharp B_{(\h,L)}[[\hbar]].$$
Then from \cite{Li-smash}, for any $V_{(\h,L)}$-module $(W,Y_W)$,
there exists a $V_{(\h,L)}[[\hbar]]\sharp B_{(\h,L)}[[\hbar]]$-module
structure $Y_W^{\eta}(\cdot,x)$ on $W[[\hbar]]$ such that
\begin{eqnarray*}
&&Y_W^{\eta}(a,x)=Y_W(a,x)+\Phi(\eta'(a,x)) \quad \text{ for }a\in \h,\\
&&Y_W^{\eta}(e_{\alpha},x)=Y_W(e_\al,x)\exp(\Phi(\eta(\alpha,x)))\quad \text{ for }\alpha\in L.
\end{eqnarray*}
Consequently, there exists a $V_{(\h,L)}[[\hbar]]^\eta$-module structure $Y_W^{\eta}(\cdot,x)$ on $W[[\hbar]]$,
satisfying the required conditions.
\end{proof}

For $f(x)\in \C((x))[[\hbar]]$, let $f(x)^{\pm}$ denote the regular/singular parts of $f(x)$.
Furthermore, for $a\in \h$, let $\eta(a,x)^{\pm}$ denote the regular/singular parts of $\eta(a,x)$.

\begin{prop}\label{lem:commutator-V-L-h-eta}
The following relations hold on $V_{(\h,L)}[[\hbar]]^{\eta}$  for $\al, \be\in L$:
\begin{align}
&[Y_{L}^\eta(\al,x_1),Y_{L}^\eta(\be,x_2)]
=\<\al,\be\>\frac{\partial}{\partial x_2}x_1\inverse\delta\!\(\frac{x_2}{x_1}\)\label{new-h-h-relations}\\
&\quad\quad \quad \quad  -\<\eta''(\al,x_1-x_2),\be\> +\<\eta''(\be,x_2-x_1),\al\>, \nonumber\\
&[Y_{L}^\eta(\al,x_1),Y_{L}^\eta(e_\be,x_2)]
=\<\be,\al\>Y_{L}^\eta(e_\be,x_2)x_1\inverse\delta\!\(\frac{x_2}{x_1}\)\label{new-h-e-beta-relations}\\
&\quad\quad \quad \quad + \(\<\eta'(\al,x_1-x_2),\be\>+ \<\eta'(\beta, x_2-x_1),\al\>\)
    Y_{L}^\eta(e_\be,x_2),\nonumber\\
&e^{-\<\eta(\al,x_1-x_2),\be\>}(x_1-x_2)^{-\<\al,\be\>-1}
Y_{L}^\eta(e_\al,x_1)Y_{L}^\eta(e_\be,x_2)\label{Yeta-ealpha=ebeta}\\
&\quad - e^{-\<\eta(\be,x_2-x_1),\al\>}
  (-x_2+x_1)^{-\<\al,\be\>-1}Y_{L}^\eta(e_\be,x_2)Y_{L}^\eta(e_\al,x_1)\nonumber\\
=\ &  \varepsilon(\al,\be)Y_{L}^\eta(e_{\al+\be},x_2)
    x_1\inverse\delta\!\(\frac{x_2}{x_1}\)\!.\nonumber
\end{align}
\end{prop}

\begin{prop}\label{classical-limit-qva}
Let $\eta(\cdot,x)\in\Hom(\h, \h\ot \C((x))[[\hbar]])^0$ and set $\eta_0(\al,x)=\eta(\al,x)|_{\hbar=0}$.
Then $V_{(\h,L)}[[\hbar]]^\eta/\hbar V_{(\h,L)}[[\hbar]]^\eta\cong V_{(\h,L)}^{\eta_0}$ as nonlocal vertex algebras.
Furthermore, if
\begin{align}
\<\eta_0(\al,x),\be\>=\<\al,\eta_0(\be,-x)\>\quad \te{ for }\al,\be\in \h,\label{zero-condition2}
\end{align}
then $V_{(\h,L)}^{\eta_{0}}\cong V_{(\h,L)}$.
\end{prop}

\begin{rem}\label{non-degenerate}
{\em Recall that $V_{(\h,L)}$ is a simple vertex algebra (see \cite{LX}). By \cite{Li-nondeg}, $V_{(\h,L)}$ is non-degenerate,
and hence $V_L$ as a vertex subalgebra of $V_{(\h,L)}$ is also non-degenerate.}
\end{rem}

The following structural results on $V_{(\h,L)}[[\hbar]]^\eta$ were also obtained therein:

\begin{prop}\label{VL-eta-S(X)-general}
Let $\eta(\cdot,x)\in {\rm Hom}(\h, \h\otimes \C((x))[[\hbar]])^{0}$.
Then $V_{(\h,L)}[[\hbar]]^{\eta}$ is non-degenerate with a uniquely determined $\SY$-locality operator $\SY(x)$.
Furthermore, the following relations hold for $a,b\in \h,\ \al,\be\in L$:
\begin{eqnarray}
&&\mathcal{S}(x)(b\otimes a)=b\ot a\ot 1+{\bf 1}\ot {\bf 1}\ot (\<\eta''(b,x),a\>-\<\eta''(a,-x),b\>),\\
&&\mathcal{S}(x)(e_{\be}\ot a)=e_\be\ot a\ot 1+e_\be \ot \vac \ot (\<\eta'(\be,x),a\>+\<\eta'(a,-x),\be\>),\\
&&\mathcal{S}(x)(e_{\beta}\otimes e_{\al})=(e_{\be}\otimes e_{\al})\ot e^{\<\beta,\eta(\al,-x)\>-\<\al,\eta(\be,x)\>},
\end{eqnarray}
and
\begin{align}
&Y_L^{\eta}(a,x)^{-}b=\(\<a,b\>x^{-2}-\<\eta''(a,x),b\>^{-}\)\vac,\label{eta-al-be-sing}\\
&Y_L^{\eta}(a,x)^{-}e_{\beta}=\(\<a,\be\>x^{-1}+\<\eta'(a,x),\be\>^{-}\)e_{\be},\label{eta-al-ebe-sing}\\
&Y_L^{\eta}(e_\al,x)e_\be =\varepsilon(\al,\be)x^{\<\al,\be\>}e^{\<\be,\eta(\al,x)\>}E^{-}(-\al,x)e_{\al+\be}.
\label{Y-eta-e-alha-e-be}
\end{align}
\end{prop}

Next, we introduce some particular maps $\eta(\cdot,x)\in {\rm Hom}(\h, \h\otimes \C((x))[[\hbar]])^{0}$.
Recall that $\h_0$ is a subspace of $\h$ with a basis $\{ \alpha_i\ |\ i\in I\}$.
Let $\lambda_i\in\h$ for $i\in I$ such that
\begin{align}
\<\lambda_i,\al_j\>=\delta_{i,j}\quad \te{for } i,j\in I.
\end{align}

\begin{rem}\label{center-C}
{\em If $A$ is of finite type, then $\h=\h_0$ and $\{\lambda_i\ |\ i\in I\}$
 (uniquely determined) is the basis of $\h$ dual to $\{ \alpha_i\ |\ i\in I\}$.
In general, $\lambda_i$ for $i\in I$ are unique modulo $\mathfrak{c}:=\{ a\in \h\ |\  \<a,\h_0\>=0\}$,
which equals the center of Lie algebra $\g(A)$. One can see that $\mathfrak{c}$ lies in the centralizer of
$V_{L}$ in vertex algebra $V_{(\h,L)}$. }
 \end{rem}

Fix {\em a choice} of $\lambda_i\in\h$ for $i\in I$.
Note that for $m\in \Z$,  $[m]_{e^{x}}e^{x}-m\in x \C[[x]]$, so
\begin{align}
\left([m]_{q^{\partial_z}}q^{\partial_z}-m\right)\log z\in \hbar z^{-1}\C[z^{-1}][[\hbar]].
\end{align}

\begin{de}
 Define a linear map $\eta_{\Y}(\cdot,z):\  \h_{0} \rightarrow \h\otimes \hbar z^{-1}\C[z^{-1}][[\hbar]]$  by
\begin{align}
\eta_\Y(\alpha_i,z)
=\sum_{j\in I}\lambda_j\otimes  \([a_{i,j}]_{q^{\partial_z}}q^{\partial_z}-a_{i,j}\)\log z\quad \text{ for }i\in I.
\end{align}
\end{de}

\begin{prop}
There exists an $\hbar$-adic weak quantum vertex algebra structure on $V_L[[\hbar]]$,
where the vertex operator map denoted by $Y_L^{\eta_\Y}(\cdot,x)$ is uniquely determined by
\begin{eqnarray}
&&Y_L^{\eta_\Y}(a,x)=Y(a,x)+\Phi(\eta_{\Y}'(a,x)) \quad \text{ for }a\in \h_0,\\
&&Y_L^{\eta_\Y}(e_{\alpha},x)=Y(e_\al,x)\exp(\Phi(\eta_\Y(\alpha,x)))\quad \text{ for }\alpha\in L.
\end{eqnarray}
Denote this $\hbar$-adic weak quantum vertex algebra by $V_L[[\hbar]]^{\eta_\Y}$.
Then $V_L[[\hbar]]^{\eta_\Y}$ is independent of the choice of $\lambda_i$ for $i\in I$ and
it is a non-degenerate $\hbar$-adic quantum vertex algebra  with
$V_L[[\hbar]]^{\eta_\Y}/\hbar V_L[[\hbar]]^{\eta_\Y}=V_L$.
\end{prop}

\begin{proof} Notice that $\eta_\Y(\cdot,z)|_{\hbar=0}=0$.
Let $p_0: \h\rightarrow \h_0$ be any vector space projection of $\h$ onto $\h_0$,
i.e., $p_0\in \Hom (\h,h_0)$ with $p_0|_{\h_0}=1$.
Then we have a linear map
\begin{align}
\eta(\cdot,z):=\eta_\Y(\cdot,z)\circ p_0\in {\rm Hom}(\h, \h\otimes \C((z))[[\hbar]])^{0}
\end{align}
with $\eta(\cdot,z)|_{\hbar=0}=0$.
By Theorem \ref{thm:qlva} and Proposition \ref{classical-limit-qva},
we have an $\hbar$-adic quantum vertex algebra
$V_{(\h,L)}[[\hbar]]^{\eta}$ with $V_{(\h,L)}[[\hbar]]^{\eta}/\hbar V_{(\h,L)}[[\hbar]]^{\eta}\simeq V_{(\h,L)}$.
From (\ref{Y-L-eta-1}), (\ref{Y-L-eta-2}), and (\ref{Phi-G(z)-V-L}),
we see that the $\C[[\hbar]]$-submodule $V_L[[\hbar]]$ of $V_{(\h,L)}[[\hbar]]^{\eta}$
is an $\hbar$-adic nonlocal vertex subalgebra, which is denoted by $V_L[[\hbar]]^{\eta}$.
It is clear that $V_L[[\hbar]]^{\eta}$ is independent of the choice of $p_0$.
On the other hand, for $a\in \mathfrak{c}=\{ c\in \h\ |\  \<c,\h_0\>=0\}$, we have $a(n)V_L=0$ for $n\ge 0$ (as $a(n){\bf 1}=0$),
so that $\Phi(a, f(x))=0$ on $V_L[[\hbar]]$ for any $f(x)\in \C((x))[[\hbar]]$.
Then by Remark \ref{center-C}, $V_L[[\hbar]]^{\eta}$ is independent of the choice of $\lambda_i$ for $i\in I$.
Using the relations in Proposition \ref{lem:commutator-V-L-h-eta}
(or  from the explicit expression of the $\SY$-locality operator in \cite{JKLT-22}),
we see that $V_L[[\hbar]]^{\eta}$ is an $\hbar$-adic weak quantum vertex algebra.
From Proposition \ref{classical-limit-qva}, we have $V_L[[\hbar]]^{\eta}/\hbar V_L[[\hbar]]^{\eta}=V_L$,
which is non-degenerate. Therefore, $V_L[[\hbar]]^{\eta}$ is a non-degenerate $\hbar$-adic quantum vertex algebra.
\end{proof}

 Let $a\in \h_0\subset V_L$. Write
\begin{align}
Y_L^{\eta_\Y}(a,x)=\sum_{n\in \Z}a^{\Y}(n)x^{-n-1}
\end{align}
on $V_L[[\hbar]]^{\eta_\Y}$.
The following is a technical result which we shall need:

\begin{lem}\label{E-Eeta-equality}
Let $z\in \hbar\C[[\hbar]]$. For $\alpha\in L$, set
 \begin{align}
 E_{\eta_\Y}^{\pm }(\al,z)=\exp\!\(\sum_{n\ge 1}\frac{\al^{\Y}(\pm n)}{\pm n}z^{\mp n}\)\!
 \end{align}
 (cf. (\ref{E-pm-h,z})). Then
  \begin{align}\label{E-Eeta-vacuum-equality}
 E_{\eta_\Y}^{-}(\pm \al_i,z){\bf 1}=E^{-}(\pm \al_i,z){\bf 1}\quad \text{ for }i\in I.
 \end{align}
\end{lem}

\begin{proof} Let $\alpha\in L$.  Recall that $Y_{L}^{\eta_\Y}(\al,x)=Y(\al,x)+\Phi(\eta_{\Y}'(\al,x))$.
We have
\begin{align}\label{Phi-bracket-formula}
\sum_{n\ge 1}\frac{\al^{\Y}(-n)}{-n}z^{n}&=\sum_{n\ge 1}\frac{\al(-n)}{-n}z^{n}+\sum_{n\ge 1} \frac{1}{-n}z^{n}\Res_x
x^{-n}\Phi(\eta_{\Y}'(\al,x))\\
&=\sum_{n\ge 1}\frac{\al(-n)}{-n}z^{n}-\sum_{n\ge 1}z^{n}\Res_x
x^{-n-1}\Phi(\eta_\Y(\al,x))\nonumber\\
&=\sum_{n\ge 1}\frac{\al(-n)}{-n}z^{n}+\Res_x\left(\frac{1}{x}-\frac{1}{x-z}\right)
\Phi(\eta_\Y(\al,x)).\nonumber
\end{align}
Note that
$$\left[\Phi(\eta_\Y(\alpha,x)),\sum_{n\ge 1}\frac{\alpha(-n)}{-n}z^n\right]=\<\eta_\Y(\alpha,x)-\eta_\Y(\alpha,x-z),\alpha\>.$$
Then
\begin{align}
&\Res_x\left(\frac{1}{x}-\frac{1}{x-z}\right)\left[\Phi(\eta_\Y(\alpha,x)),\sum_{n\ge 1}\frac{\alpha(-n)}{-n}z^n\right]\\
=\ &\Res_x\left(\frac{1}{x}-\frac{1}{x-z}\right)(\<\eta_\Y(\alpha,x)-\eta_\Y(\alpha,x-z),\alpha\>).\nonumber
\end{align}

Now, consider the case with $\alpha=\pm \alpha_i$. From the definition of $\eta_\Y(\cdot,x)$, we have
$$\<\eta_\Y(\alpha_i,x),\alpha_i\>=(q^{2\partial_x}-1)\log x\in \hbar x^{-1}\C[x^{-1}][[\hbar]],$$
which implies that $\<\eta_\Y(\alpha_i,x),\alpha_i\>-\<\eta_\Y(\alpha_i,x-z),\alpha_i\>\in  \hbar x^{-1}\C[x^{-1}][[\hbar,z]]$, so that
\begin{align}\label{commutativity-sum}
&\Res_x\left(\frac{1}{x}-\frac{1}{x-z}\right)\left[\Phi(\eta_\Y(\alpha_i,x)),\sum_{n\ge 1}\frac{\alpha_i(-n)}{-n}z^n\right]\\
=\ & \Res_x\left(\frac{1}{x}-\frac{1}{x-z}\right)\left(\<\eta_\Y(\alpha_i,x)-\eta_\Y(\alpha_i,x-z),\alpha_i\>\right)=0.\nonumber
\end{align}
Notice that $\Phi(\eta_\Y(\al_i,x)){\bf 1}=0$. Then using (\ref{Phi-bracket-formula}) and (\ref{commutativity-sum})
we get (\ref{E-Eeta-vacuum-equality}).
\end{proof}

For convenience, recall  from \cite{KL1} and \cite{JKLT-22} the following simple fact:

\begin{lem}\label{lem-sing-fact}
Let $b\in \C,\ F(x,\hbar)\in \C[[x,\hbar]]$. Then
\begin{eqnarray}
&&{\rm Sing}_x (x-b\hbar)^{-1}F(x,\hbar)=(x-b\hbar)^{-1}F(b\hbar, \hbar),\label{singular-fact}\\
&&\Res_x(x-b \hbar)^{-1}F(x,\hbar)
=F(b \hbar,\hbar).
\end{eqnarray}
\end{lem}

Now, we give more explicit commutation relations for $V_L[[\hbar]]^{\eta_\Y}$.

\begin{prop}\label{lattice-main-properties}
The following relations hold on $V_L[[\hbar]]^{\eta_\Y}$ for $i,j\in I$:
\begin{align*}
&[Y_L^{\eta_\Y}(\al_i,z),Y_L^{\eta_\Y}(\al_j,w)]=[a_{i,j}]_{q^{\partial_{w}}}
       \!\( (z-w+\hbar )^{-2}-(w-z+\hbar)^{-2}\)\!,\\
&[Y_L^{\eta_\Y}(\al_i,z),Y_L^{\eta_\Y}(e_{\pm\al_j},w)]
=\pm Y_L^{\eta_\Y}(e_{\pm \al_j},w)\;[a_{i,j}]_{q^{\partial_{w}}}\!\((z-w+\hbar )^{-1}+(w-z+\hbar )^{-1}\)\!,\\
&(z-w-a_{i,j}\hbar)Y_L^{\eta_\Y}(e_{\pm \al_i},z)Y_L^{\eta_\Y}(e_{\pm \al_j},w)
=(z-w+a_{i,j}\hbar)Y_L^{\eta_\Y}(e_{\pm \al_j},w)Y_L^{\eta_\Y}(e_{\pm \al_i},z),\\
&Y_L^{\eta_\Y}(e_{\pm \al_i},z)Y_L^{\eta_\Y}(e_{\pm \al_j},w)
=
Y_L^{\eta_\Y}(e_{\pm \al_j},w)Y_L^{\eta_\Y}(e_{\pm \al_i},z)\ \ \text{ if }a_{i,j}= 0,\\
&Y_L^{\eta_\Y}(e_{\al_i},z)Y_L^{\eta_\Y}(e_{-\al_j},w)-\( \frac{w-z+a_{i,j}\hbar}{w-z-a_{i,j}\hbar}\)\!Y_L^{\eta_\Y}(e_{-\al_j},w)
Y_L^{\eta_\Y}(e_{\al_i},z)\nonumber\\
&\quad=\delta_{i,j}\varepsilon(\al_i,-\al_i)\frac{1}{2\hbar}
\(\!z\inv\delta\!\(\frac{w}{z}\)- 
   Y_L^{\eta_\Y}\(E_{\eta_\Y}^{-}(-\al_i,-2\hbar){\bf 1},w\)
    z\inv\delta\!\(\frac{w-2\hbar }{z}\)\!\)\!.
\end{align*}
\end{prop}

\begin{proof} Let $i,j\in I$. Note that $\< \eta_\Y(\alpha_i,z),\al_j\>=\([a_{i,j}]_{q^{\partial_z}}q^{\partial_z}-a_{i,j}\)\log z$.
Since
\begin{align*}
\<\eta_\Y''(\al_i,z),\al_j\>=-[a_{i,j}]_{q^{\partial_z}}q^{\partial_z}z^{-2}+a_{i,j}z^{-2},
\end{align*}
$a_{i,j}=a_{j,i}$, and $[a_{i,j}]_{q^{\partial_z}}q^{\partial_z}(z-w)^{-2}=[a_{i,j}]_{q^{\partial_w}}(z-w+\hbar)^{-2}$,
by Proposition \ref{lem:commutator-V-L-h-eta} we get
\begin{align*}
&[Y_{L}^{\eta_\Y}(\al_i,z),Y_{L}^{\eta_\Y}(\al_j,w)]\\
=\ &\<\al_i,\al_j\>\frac{\partial}{\partial w}z\inverse\delta\!\(\frac{w}{z}\)
 -\<\eta_\Y''(\al_i,z-w),\al_j\> +\<\eta_\Y''(\al_j,w-z),\al_i\>\\
 =\ &a_{i,j}\frac{\partial}{\partial w}z\inverse\delta\!\(\frac{w}{z}\)+[a_{i,j}]_{q^{\partial_z}}(z-w+\hbar)^{-2}-a_{i,j}(z-w)^{-2}\\
 &\  \ -[a_{j,i}]_{q^{\partial_w}}(w-z+\hbar)^{-2}+a_{j,i}(w-z)^{-2} \\
  =\ &[a_{i,j}]_{q^{\partial_w}}\((z-w+\hbar)^{-2}-(w-z+\hbar)^{-2}\)\!.
\end{align*}
Similarly, as
\begin{align*}
&\< \eta_\Y'(\alpha_i,z-w),\al_j\>+\< \eta_\Y'(\alpha_j,w-z),\al_i\>\\
=\ &([a_{i,j}]_{q^{\partial_z}}q^{\partial_z}-a_{i,j})(z-w)^{-1}+([a_{j,i}]_{q^{\partial_w}}q^{\partial_w}-a_{j,i})(w-z)^{-1}\\
=\ &-a_{i,j}z^{-1}\delta\!\(\frac{w}{z}\)+[a_{i,j}]_{q^{\partial_w}}\((z-w+\hbar)^{-1}+(w-z+\hbar)^{-1} \)\!,
\end{align*}
we have
\begin{align*}
[Y_L^{\eta_\Y}(\al_i,z),Y_L^{\eta_\Y}(e_{\pm\al_j},w)]
=\pm Y_L^{\eta_\Y}(e_{\pm\al_j},w)\;[a_{i,j}]_{q^{\partial_w}}\!\((z-w+\hbar)^{-1}+(w-z+\hbar)^{-1} \)\!.
\end{align*}

Let $\al=\delta \al_i$ and $\beta=\delta' \al_j$ with $\delta,\delta'\in \{\pm 1\}$. Using (\ref{Y-eta-e-alha-e-be})
we get
\begin{align*}
Y_L^{\eta_\Y}(e_{\delta\al_i},x)e_{\delta_\Y'\al_j}
 =\ & \varepsilon(\delta\al_i,\delta'\al_j)x^{\delta\delta'\<\al_i,\al_j\>}e^{\delta\delta'\<\al_j,\eta(\al_i,x)\>}
 E^{-}(-\delta\al_i,x)e_{\delta\al_i+\delta'\al_j} \\
 =\ & \varepsilon(\delta\al_i,\delta'\al_j)x^{\delta\delta' a_{i,j}} e^{\delta\delta'([a_{i,j}]_{q^{\partial_x}}q^{\partial_x}-a_{i,j})\log x}
 E^{-}(-\delta\al_i,x)e_{\delta\al_i+\delta'\al_j}.\ \
\end{align*}
With $a_{i,j}\in \{0, 2, -1\}$, more explicitly we have
\begin{align*}
Y_L^{\eta_\Y}(e_{\delta\al_i},x)e_{\delta'\al_j}=\begin{cases}
\varepsilon(\delta\al_i,\delta'\al_j)E^{-}(-\delta\al_i,x)e_{\delta\al_i+\delta'\al_j} \ \ &\text{ if }a_{i,j}=0\\
\varepsilon(\delta\al_i,\delta'\al_j)(x(x+2\hbar))^{\delta\delta'}E^{-}(-\delta\al_i,x)e_{\delta\al_i+\delta'\al_j}\ \ &\text{ if }a_{i,j}=2\\
\varepsilon(\delta\al_i,\delta'\al_j) (x+\hbar)^{-\delta\delta'}
E^{-}(-\delta\al_i,x)e_{\delta\al_i+\delta'\al_j}\ \ &\text{ if }a_{i,j}=-1.
 \end{cases}
\end{align*}
On the other hand, by Proposition \ref{VL-eta-S(X)-general} we have
\begin{align*}
\mathcal{S}(x)(e_{\delta' \al_j}\otimes e_{\delta\al_i})&=(e_{\delta' \al_j}\otimes e_{\delta\al_i})
\ot e^{\delta\delta' (\<\al_j,\eta(\al_i,-x)\>- \<\al_i,\eta(\al_j,x)\>)}\\
&=(e_{\delta' \al_j}\otimes e_{\delta\al_i})\ot e^{-\delta\delta' (q^{a_{i,j}\partial_x}- q^{-a_{i,j}\partial_x})\log x}\\
&=(e_{\delta' \al_j}\otimes e_{\delta\al_i})\ot \left(\frac{x-a_{i,j}\hbar}{x+a_{i,j}\hbar}\right)^{\delta\delta'}\!.
\end{align*}

Assume $a_{i,j}= 0$. Then
\begin{align*}
&Y_L^{\eta_\Y}(e_{\pm \al_i},x)e_{\pm \al_j}\in V_L[[\hbar]]^{\eta_\Y}[[x]],\\
&\mathcal{S}(x)(e_{\pm \al_j}\otimes e_{\pm \al_i})
=e_{\pm \al_j}\otimes e_{\pm \al_i}.
\end{align*}
Using these facts and the $\SY$-commutator formula we get
\begin{align}
Y_L^{\eta_\Y}(e_{\pm \al_i},z)Y_L^{\eta_\Y}(e_{\pm \al_j},w)
=
Y_L^{\eta_\Y}(e_{\pm \al_j},w)Y_L^{\eta_\Y}(e_{\pm \al_i},z).
\end{align}

Assume $a_{i,j}=-1$. In this case, we have
\begin{align}
&Y_L^{\eta_\Y}(e_{\pm \al_i},x)e_{\pm \al_j}=\varepsilon(\pm\al_i,\pm \al_j)(x+\hbar)^{-1}E^{-}(\mp \al_i,x)e_{\pm (\al_i+\al_j)},
\label{singular-ij-proof}\\
&\mathcal{S}(x)(e_{\pm \al_j}\otimes e_{\pm \al_i})
=(e_{\pm \al_j}\otimes e_{\pm \al_i})\ot \left(\frac{x+\hbar}{x-\hbar}\right)\!,\nonumber
\end{align}
so
\begin{align*}
&Y_L^{\eta_\Y}(e_{\pm \al_i},z)Y_L^{\eta_\Y}(e_{\pm \al_j},w)
-\left(\frac{w-z+\hbar}{w-z-\hbar}\right)\!Y_L^{\eta_\Y}(e_{\pm \al_j},w)Y_L^{\eta_\Y}(e_{\pm \al_i},z)\\
=\ &\Res_x z^{-1}\delta\!\left(\frac{w+x}{z}\right)\!Y_L^{\eta_\Y}\!\(Y_L^{\eta_\Y}(e_{\pm \al_i},x)e_{\pm \al_j},w\)\\
=\ &\Res_x \varepsilon(\pm\al_i,\pm\al_j)(x+\hbar)^{-1}z^{-1}\delta\!\left(\frac{w+x}{z}\right)\!
Y_L^{\eta_\Y}\!\(E^{-}(\mp\al_i,x)e_{\pm(\al_i+\al_j)},w\)\\
=\ &z^{-1}\delta\!\left(\frac{w-\hbar}{z}\right)\!
\varepsilon(\pm\al_i,\pm\al_j)Y_L^{\eta_\Y}\!\(E^{-}(\mp\al_i,-\hbar)e_{\pm(\al_i+\al_j)},w\)\!.
\end{align*}
Multiplying both sides by $(w-z-\hbar)$ we get the desired commutation relation
\begin{align*}
(z-w+\hbar)Y_L^{\eta_\Y}(e_{\pm \al_i},z)Y_L^{\eta_\Y}(e_{\pm \al_j},w)
=(z-w-\hbar)Y_L^{\eta_\Y}(e_{\pm \al_j},w)Y_L^{\eta_\Y}(e_{\pm \al_i},z).
\end{align*}

For the last assertion, we have
\begin{align*}
\mathcal{S}(x)(e_{\mp \al_j}\otimes e_{\pm \al_i})
=(e_{\mp \al_j}\otimes e_{\pm \al_i})\ot \left(\frac{x+a_{i,j}\hbar}{x-a_{i,j}\hbar}\right)\!.
\end{align*}
If $i\ne j$, then $a_{i,j}=0$, or $-1$. In this case, $Y_L^{\eta_\Y}(e_{\al_i},x)e_{-\al_j}$ is regular in $x$, so we get
$$Y_L^{\eta_\Y}(e_{\al_i},z)Y_L^{\eta}(e_{-\al_j},w)
-\left(\frac{w-z+a_{i,j}\hbar}{w-z-a_{i,j}\hbar}\right)Y_L^{\eta_\Y}(e_{- \al_j},w)Y_L^{\eta_\Y}(e_{\al_i},z)=0.$$
For $i=j$, we have
\begin{align*}
Y_L^{\eta_\Y}(e_{\al_i},x)e_{-\al_i}
= \frac{\varepsilon(\al_i,-\al_i)}{x(x+2\hbar)}E^{-}(-\al_i,x){\bf 1}
=\varepsilon(\al_i,-\al_i)\frac{1}{2\hbar}\!\left( \frac{1}{x}-\frac{1}{x+2\hbar}\right)\!E^{-}(-\al_i,x){\bf 1}.
\end{align*}
Then using Lemma \ref{lem-sing-fact} we obtain
\begin{align*}
&Y_L^{\eta_\Y}(e_{\al_i},z)Y_L^{\eta_\Y}(e_{-\al_i},w)
-\left(\frac{w-z+2\hbar}{w-z-2\hbar}\right)Y_L^{\eta_\Y}(e_{- \al_i},w)Y_L^{\eta_\Y}(e_{\al_i},z)\\
=\ &\Res_{x}Y_L^{\eta_\Y}(Y_L^{\eta_\Y}(e_{\al_i},x)e_{-\al_i},w)z^{-1}\delta\!\left(\frac{w+x}{z}\right)\nonumber\\
=\ &\Res_{x}\varepsilon(\al_i,-\al_i)\frac{1}{2\hbar}\left( \frac{1}{x}-\frac{1}{x+2\hbar}\right)
Y_L^{\eta_\Y}(E^{-}(-\al_i,x){\bf 1},w)z^{-1}\delta\!\left(\frac{w+x}{z}\right)\nonumber\\
=\ &\varepsilon(\al_i,-\al_i)\frac{1}{2\hbar}\left(\!z^{-1}\delta\!\left(\frac{w}{z}\right)-
Y_L^{\eta_\Y}(E^{-}(-\al_i,-2\hbar){\bf 1},w)z^{-1}\delta\!\left(\frac{w-2\hbar}{z}\right)\! \right)\!.\nonumber
\end{align*}
Note that by Lemma \ref{E-Eeta-equality} we have
$E_{\eta_\Y}^{-}(\mp\al_i,-2\hbar)\vac=E^{-}(\mp \al_i,-2\hbar)\vac$. This proves the last relation of this proposition,
completing the proof.
\end{proof}

Furthermore, we have:

\begin{prop}\label{lem:Y-eta-rel-5}
Assume $i,j\in I$ such that $a_{i,j}=\<\al_i,\al_j\>=-1$. Then
\begin{align}\label{sing-z-sing-x}
\Sing_{x}\Sing_{z}Y_L^{\eta_\Y}(e_{\pm \al_i},x)Y_L^{\eta_\Y}(e_{\pm\al_i},z)e_{\pm\al_j}=0.
\end{align}
Furthermore,  the Serre relations
for $(Y_L^{\eta_\Y}(e_{\pm \al_i},z_1),Y_L^{\eta_\Y}(e_{\pm \al_i},z_2),Y_L^{\eta_\Y}(e_{\pm \al_j},x))$ hold.
\end{prop}

\begin{proof} From Proposition \ref{lattice-main-properties} we have
\begin{align*}
&(z-w-2\hbar)Y_L^{\eta_\Y}(e_{\pm \al_i},z)Y_L^{\eta_\Y}(e_{\pm \al_i},w)
=(z-w+2\hbar)Y_L^{\eta_\Y}(e_{\pm \al_i},w)Y_L^{\eta_\Y}(e_{\pm \al_i},z),\\
&(z-w+\hbar)Y_L^{\eta_\Y}(e_{\pm \al_i},z)Y_L^{\eta_\Y}(e_{\pm \al_j},w)
=(z-w-\hbar)Y_L^{\eta_\Y}(e_{\pm \al_j},w)Y_L^{\eta_\Y}(e_{\pm \al_i},z).
\end{align*}
Then in view of Theorem \ref{Serre-relation-equivalence-va} it suffices to show that (\ref{sing-z-sing-x}) holds.

Recall (\ref{singular-ij-proof}):
\begin{align*}
Y_L^{\eta_\Y}(e_{\pm\al_i},z)e_{\pm\al_j}
=\varepsilon(\pm \al_i,\pm \al_j)(z+\hbar)^{-1}E^{-}(\mp \al_i,z)e_{\pm (\al_i+\al_j)}.
\end{align*}
On the other hand, using (\ref{Y-eta-e-alha-e-be}) we get
\begin{align*}
  Y_L^{\eta_\Y}(e_{\pm\al_i},x)e_{\pm(\al_i+\al_j)}
  =\varepsilon(\pm \al_i,\pm (\al_i+\al_j)) \frac{x(x+2\hbar)}{x+\hbar} E^{-}(\mp \al_i,x)e_{\pm(2\al_i+\al_j)}.
\end{align*}
From \cite{JKLT-22} (Section 6), for $\al,\be\in L$ we have
\begin{align}\label{E-Y-comm}
&E^{-}(\be,z)Y_L^{\eta_\Y}(e_{\al},x)\\
=\ &Y_L^{\eta_\Y}(e_{\al},x)E^{-}(\be,z)\(1-\frac{z}{x}\)^{\<\al,\be\>}\exp (\<\eta_\Y(\al,x-z)-\eta_\Y(\al,x),\be\>).\nonumber
\end{align}
Noticing that
$$e^{\<\al_i,\eta_\Y(\al_i,x)\>}=x^{-1}(x+2\hbar),\quad e^{\<\al_j,\eta_\Y(\al_i,x)\>}=x(x+\hbar)^{-1},$$
and then using (\ref{E-Y-comm}) we get
\begin{align*}
 Y_L^{\eta_\Y}(e_{\pm\al_i},x)E^{-}(\mp \al_i,z)
 = E^{-}(\mp \al_i,z)Y_L^{\eta_\Y}(e_{\pm\al_i},x)\frac{(x-z)(x+2\hbar-z)}{x(x+2\hbar)}.
\end{align*}
Using all of these relations, we obtain
\begin{align*}
&Y_L^{\eta_\Y}(e_{\pm\al_i},x)Y_L^{\eta_\Y}(e_{\pm\al_i},z)e_{\pm\al_j}\\
=\ &\varepsilon(\pm \al_i, \pm \al_j)(z+\hbar)^{-1}Y_L^{\eta_\Y}(e_{\pm\al_i},x)E^{-}(\mp \al_i,z)e_{\pm (\al_i+\al_j)}\\
=\ &\varepsilon(\pm \al_i,\pm \al_j)(z+\hbar)^{-1}\frac{(x-z)(x+2\hbar-z)}{x(x+2\hbar)}E^{-}(\mp \al_i,z)
Y_L^{\eta_\Y}(e_{\pm\al_i},x)e_{\pm (\al_i+\al_j)}\\
=\ &\mu \frac{(x-z)(x+2\hbar-z)}{(x+\hbar)(z+\hbar)} E^{-}(\mp \al_i,z)E^{-}(\mp \al_i,x)e_{\pm (2\al_i+\al_j)},
\end{align*}
where $\mu=\varepsilon(\pm \al_i,\pm \al_j)\varepsilon(\pm \al_i,\pm (\al_i+\al_j))$.
Then by Lemma \ref{lem-sing-fact} we get
\begin{align*}
{\rm Sing}_zY_L^{\eta_\Y}(e_{\pm\al_i},x)Y_L^{\eta_\Y}(e_{\pm\al_i},z)e_{\pm\al_j}
=\mu
\!\(\!\frac{x+3\hbar}{z+\hbar}\!\)\!
E^{-}(\mp \al_i,-\hbar)E^{-}(\mp \al_i,x)e_{\pm (2\al_i+\al_j)},
\end{align*}
which is regular in $x$.  This proves that (\ref{sing-z-sing-x}) holds, concluding the proof.
\end{proof}

Now, we are in a position to present the main result of this paper.

\begin{thm}\label{V-module-DY-module}
Let $(W,Y_W)$ be any $V_L[[\hbar]]^{\eta_\Y}$-module.
Then there exists a restricted $\wh{\mathcal{DY}}(A)$-module structure on $W$  of level one,
which is uniquely determined by
\begin{align}\label{h=Y-W}
h_{i,\Y}(z)=Y_W(\al_i,z),\ x_{i,\Y}^{+}(z)=Y_W(e_{\al_i},z),\
x_{i,\Y}^{-}(z)=\varepsilon(\al_i,-\al_i)^{-1}Y_W(e_{-\al_i},z)
\end{align}
for $i\in I$.
\end{thm}

\begin{proof} We here apply Theorem \ref{thm-main-DY}.
Note that all the commutation relations on $V_L[[\hbar]]^{\eta_\Y}$ in Proposition \ref{lattice-main-properties} also hold
on $W$. On the other hand, the required Serre relations hold on $W$
by Propositions \ref{lem:Y-eta-rel-5} and \ref{Serre-relation-abstract}.
Then it remains to establish the last commutator relation on $W$.

Let $i,j\in I$. By Propositions \ref{prop-2.25} and \ref{lattice-main-properties}, we have
\begin{align*}
&Y_W(e_{\al_i},z)Y_W(e_{-\al_j},w)-\( \frac{w-z+a_{i,j}\hbar}{w-z-a_{i,j}\hbar}\)\!Y_W(e_{-\al_j},w)
Y_W(e_{\al_i},z)\nonumber\\
&=\delta_{i,j}\varepsilon(\al_i,-\al_i)\frac{1}{2\hbar}
\(\!z\inv\delta\!\(\frac{w}{z}\)-
   Y_W\!\(E_{\eta_\Y}^{-}(-\al_i,-2\hbar){\bf 1},w\)
    z\inv\delta\!\(\frac{w-2\hbar }{z}\)\!\)\!.\nonumber
 \end{align*}

Note that
\begin{align*}
[Y_L^{\eta_\Y}(\al_i,x_1), Y_L^{\eta_\Y}(\al_i,x_2)]
=(q^{\partial_{x_2}}+q^{-\partial_{x_2}})\( (x_1-x_2+\hbar)^{-2}-(x_2-x_1+\hbar)^{-2}\)\!,
\end{align*}
which implies
\begin{align*}
[(Y_L^{\eta_\Y})^{\pm}(\al_i,x_1), (Y_L^{\eta_\Y})^{\pm}(\al_i,x_2)]& =0,\\
[(Y_L^{\eta_\Y})^{-}(\al_i,x_1), (Y_L^{\eta_\Y})^{+}(\al_i,x_2)]
&=(q^{\partial_{x_2}}+q^{-\partial_{x_2}})(x_1-x_2+\hbar)^{-2}\\
& =(x_1-x_2)^{-2}+(x_1-x_2+2\hbar)^{-2}.
\end{align*}
Set $\gamma(x)=x^{-2}+(x+2\hbar)^{-2}\in x^{-2}\C[x^{-1}][[\hbar]]$. We have
$$E_{\gamma}:=z^2\Res_x x^{-1}L(z \partial_x)L(-z \partial_x)\gamma(x)=0.$$
Then using Proposition \ref{Y-W-E(a,z)} with $a=-\al_i$ we obtain
\begin{align}
&Y_W\!\(E_{\eta_\Y}^{-}(-\al_i,-2\hbar){\bf 1},x\) \nonumber\\
=\ &\exp\!\(-2\hbar L(-2\hbar\partial_x)Y_W^{+}(\al_{i},x)\)
\exp\!\(-2\hbar L(-2\hbar\partial_x)Y_W^{-}(\al_{i},x)\)\nonumber\\
=\ &\exp\!\(-G(\partial_x)q^{-\partial_x} Y_W^{+}(\alpha_{i},x)\)
\exp\!\(-G(\partial_x)q^{-\partial_x} Y_W^{-}(\alpha_{i},x)\)\!.\nonumber
\end{align}
Now it follows that $W$ is a restricted $\wh{\mathcal{DY}}(A)$-module of level one, as desired.
\end{proof}

The following is an important property of the $\hbar$-adic quantum vertex algebra $V_L[[\hbar]]^{\eta_\Y}$:

\begin{lem}\label{a-a-relation-VL}
 For $a\in \{ e_{\pm \alpha_i}\ |\ i\in I\}$,
the following relation holds on $V_L[[\hbar]]^{\eta_\Y}$:
\begin{align}\label{level-one-intg}
&(x_1-x_2)^{-1}(x_1-x_2+2\hbar)^{-1}Y_L^{\eta_\Y}(a,x_1)Y_L^{\eta_\Y}(a,x_2)\\
=\ &(x_2-x_1)^{-1}(x_2-x_1+2\hbar)^{-1}Y_L^{\eta_\Y}(a,x_2)Y_L^{\eta_\Y}(a,x_1).\nonumber
\end{align}
\end{lem}

\begin{proof} For $\alpha\in L$, we have $Y_L^{\eta_\Y}(e_\alpha,x)=Y_L(e_\alpha,x)\exp \Phi(\eta_\Y(\alpha,x))$,
and
$$\exp \Phi(\eta_\Y(\alpha,x))Y_L(v,z)=Y_L\!\( \exp \Phi(\eta_\Y(\alpha,x-z))v,z\)\exp \Phi(\eta_\Y(\alpha,x)) $$
 (see (\ref{formula-Phi})) for $v\in V_L[[\hbar]]$.
Then for $\alpha,\beta\in L$,
\begin{align}
&Y_L^{\eta_\Y}(e_\alpha,x_1)Y_L^{\eta_\Y}(e_\beta,x_2)\\
=\ & Y_L(e_\alpha,x_1)Y_L\!\(\exp \Phi(\eta_\Y(\alpha,x_1-x_2))e_\beta,x_2\)
\exp \Phi(\eta_\Y(\alpha,x_1))\exp \Phi(\eta_\Y(\beta,x_2)).\nonumber
\end{align}
As $\Phi(\eta_\Y(\alpha,x))e_{\beta}=\<\beta, \eta_\Y(\alpha,x)\>e_{\beta}$, we have
\begin{align*}
\(\exp (\Phi(\eta_\Y(\alpha,x))\)e_{\beta}=e^{\<\beta, \eta_\Y(\alpha,x)\>}e_{\beta}.
\end{align*}
On the other hand, we have (cf. \cite{LL}, Prop. 6.5.2)
\begin{align}
(x_1-x_2)^{-\<\alpha,\beta\>}Y_L(e_{\alpha},x_1)Y_L(e_{\beta},x_2)
=(-x_2+x_1)^{-\<\alpha,\beta\>}Y_L(e_{\beta},x_2)Y_L(e_{\alpha},x_1).
\end{align}

Noticing that $\<\eta_\Y(\alpha_i,x),\alpha_i\>=(q^{2\partial_x}-1)\log x,$ we have
$$e^{\<\eta_\Y(\alpha_i,x),\alpha_i\>}=x^{-1}(x+2\hbar).$$
Taking $a=e_{\pm \alpha_i}$ with $i\in I$, we obtain
\begin{align*}
&(x_1-x_2)^{-1}(x_1-x_2+2\hbar)^{-1}Y_L^{\eta_\Y}(a,x_1)Y_L^{\eta_\Y}(a,x_2)\\
=\ & (x_1-x_2)^{-2}Y_L(a,x_1)Y_L(a,x_2)
\exp \Phi(\eta_\Y(\pm\alpha_i,x_1))\exp \Phi(\eta_\Y(\pm\alpha_i,x_2))\nonumber\\
=\ &(x_2-x_1)^{-2} Y_L(a,x_2)Y_L(a,x_1)
\exp \Phi(\eta_\Y(\pm\alpha_i,x_2))\exp \Phi(\eta_\Y(\pm\alpha_i,x_1))\nonumber\\
=\ &(x_2-x_1)^{-1}(x_2-x_1+2\hbar)^{-1}Y_L^{\eta_\Y}(a,x_2)Y_L^{\eta_\Y}(a,x_1),\nonumber
\end{align*}
as desired.
\end{proof}

Next, we give a connection of $V_L[[\hbar]]^{\eta_\Y}$ with $\V_A(1)$.

\begin{thm}\label{VL-quotient-univ}
There exists a homomorphism of $\hbar$-adic nonlocal vertex algebras
$\psi: \V_A(1)\rightarrow V_L[[\hbar]]^{\eta_\Y}$ such that
\begin{align*}
 \psi(\check{e}_i)=e_{\alpha_i},\  \psi(\check{f}_i)=\varepsilon(\al_i,-\al_i)^{-1}e_{-\alpha_i},\  \psi(\check{h}_i)=\alpha_i\ \ \text{ for }i\in I.
 \end{align*}
Moreover, $\psi$ is onto and $\ker (\psi)$ contains the strong ideal $\mathcal{J}$ generated by relations
\begin{align}\label{VL-defining-relation}
&(x_1-x_2)^{-1}(x_1-x_2+2\hbar)^{-1}Y(a,x_1)Y(a,x_2)\\
=\ & (x_2-x_1)^{-1}(x_2-x_1+2\hbar)^{-1}Y(a,x_2)Y(a,x_1)\nonumber
\end{align}
for $a\in \{ \check{e}_i,\check{f}_i\ |\ i\in I\}$. Furthermore, if $A$ is of finite ADE type, then
$\ker (\psi)=\mathcal{J}$. In particular, we have $V_L[[\hbar]]^{\eta_\Y}\simeq \V_A(1)/\mathcal{J}$.
\end{thm}

\begin{proof} By Theorem \ref{V-module-DY-module}, $V_L[[\hbar]]^{\eta_\Y}$
is a restricted $\wh{\mathcal{DY}}(A)$-module of level one with ${\bf 1}$ as a vacuum vector.
Then it follows that there exists a $\wh{\mathcal{DY}}(A)$-module homomorphism
$\psi: \V_A(1)\rightarrow V_L[[\hbar]]^{\eta_\Y}$
with $\psi({\bf 1})={\bf 1}$. With (\ref{h=Y-W}), we have
$$\psi(Y(\check{e}_i,x)v)=Y_L^{\eta_\Y}(e_{\alpha_i},x)\psi(v),\  \
\psi(Y(\check{f}_i,x)v)=\varepsilon(\al_i,-\al_i)^{-1}Y_L^{\eta_\Y}(e_{-\alpha_i},x)\psi(v),$$
$$\psi(Y(\check{h}_i,x)v)=Y_L^{\eta_\Y}(\alpha_i,x)\psi(v)\quad \text{ for all }v\in \V_A(1).$$
This in particular implies
$$\psi(\check{e}_i)=e_{\alpha_i},\  \ \psi(\check{f}_i)=\varepsilon(\al_i,-\al_i)^{-1}e_{-\alpha_i},\  \ \psi(\check{h}_i)=\alpha_i\ \ \text{ for }i\in I.$$
As $\{ \check{e}_i,\check{f}_i,\check{h}_i\ |\ i\in I\}$ generates $\V_A(1)$
as an $\hbar$-adic nonlocal vertex algebra, it follows that
 $\psi$ is a homomorphism of $\hbar$-adic nonlocal vertex algebras.

 Since $V_L[[\hbar]]^{\eta_\Y}/\hbar V_L[[\hbar]]^{\eta_\Y}=V_L$ and since
$\{e_{\pm \alpha_i}, \alpha_i\ |\ i\in I\}$ generates $V_L$ as a vertex algebra,
the induced map from $\V_A(1)/\hbar \V_A(1)$ to
$V_L[[\hbar]]^{\eta_\Y}/\hbar V_L[[\hbar]]^{\eta_\Y}$ is onto. Then $\psi$ is onto.
From Lemma \ref{a-a-relation-VL}, we have $\mathcal{J}\subset \ker \psi$. Hence
$\psi$ reduces to a surjective
$\hbar$-adic nonlocal vertex algebra homomorphism
$\bar{\psi}: \V_A(1)/{\mathcal{J}}\rightarrow V_L[[\hbar]]^{\eta_\Y}$.

For the last assertion with $A$ of finite $ADE$ type, we now prove that $\bar{\psi}$ is also injective.
Note that $\g\colon\!=\g(A)$ is a finite-dimensional simple Lie algebra.  For convenience, set
$\mathcal{V}=\V_A(1)/{\mathcal{J}}$.
By the presentation of $\wh{\g}$ in \cite{Garland},
we see that $\V_A(1)/\hbar \V_A(1)$ is
naturally a level-one  $\widehat{\g}$-module with ${\bf 1}$ as a generator.
Then ${\mathcal{V}}/ \hbar {\mathcal{V}}$ is also a level-one  $\widehat{\g}$-module with ${\bf 1}$ as a generator.
Furthermore, relation (\ref{level-one-intg}) yields
\begin{align*}
(x_1-x_2)^{-2}Y(a,x_1)Y(a,x_2)=(x_2-x_1)^{-2}Y(a,x_2)Y(a,x_1)
\end{align*}
on $\mathcal{V}/ \hbar \mathcal{V}$ for $a\in \{ \check{e}_i, \check{f}_i\ |\ i\in I\}$, which implies $[Y(a,x_1),Y(a,x_2)]=0$ and
\begin{align}
Y(a,x)Y(a,x)=0.
\end{align}
Then we have (see \cite{DL}) $\mathcal{V}/\hbar \mathcal{V}\simeq L_{\widehat{\g}}(1,0)$, the simple affine vertex operator algebra.
It follows that the $\widehat{\g}$-module homomorphism from $\mathcal{V}/\hbar {\mathcal{V}}$ to $V_L[[\hbar]]^{\eta_\Y}/\hbar V_L[[\hbar]]^{\eta_\Y}$,
reduced from $\bar{\psi}$, is injective.  Thus $\bar{\psi}$ is injective, and hence $\bar{\psi}$ is an isomorphism.
\end{proof}

Combining Theorems \ref{V-module-DY-module} and \ref{VL-quotient-univ} we immediately have:

\begin{coro}\label{all-modules}
Assume that $A$ is of finite $ADE$ type.
Then $V_L[[\hbar]]^{\eta_\Y}$-modules are exactly those level-one restricted $\wh{\mathcal{DY}}(A)$-modules on which
the following relations hold:
\begin{align}\label{add-DY-relation}
&(x_1-x_2)^{-1}(x_1-x_2+2\hbar)^{-1}e_{i,\Y}^{\pm}(x_1)e_{i,\Y}^{\pm}(x_2)\\
=\ &(x_2-x_1)^{-1}(x_2-x_1+2\hbar)^{-1}e_{i,\Y}^{\pm}(x_2)e_{i,\Y}^{\pm}(x_1)\nonumber
\end{align}
 for $i\in I$.
\end{coro}



\begin{thebibliography}{HJKOS}

\bibitem[BK]{BK}
B.~Bakalov and V.~Kac, Field algebras, {\em Internat. Math. Res. Notices} {\bf 3} (2003), 123--159.

\bibitem[B]{Bor}
 R. Borcherds, Vertex algebras, Kac-Moody algebras, and the Monster, {\em Proc. Natl. Acad. Sci. USA}
   {\bf 83} (1986), 3068--3071.

\bibitem[CJKT]{CJKT}
F.~Chen, N.~Jing, F.~Kong, and S.~Tan,
Twisted quantum affinization and quantization of extended affine
  Lie algebras, {\em Trans. Amer. Math. Soc.,} {\bf 376} (2023), 969--1039.

 \bibitem[DL]{DL}
C. Dong, J. Lepowsky, {\em Generalized Vertex Algebras and Relative Vertex Operators,}
Prog. Math. Volume 112, Birkh\"{a}user Boston, 1993.



  \bibitem[Dr]{Dr-new}
V.~Drinfeld,
A new realization of Yangians and quantized affine algebras,
{\em Soviet Math. Dokl} {\bf 36} (1988), 212--216.

\bibitem[EK]{EK-qva}
P.~Etingof and D.~Kazhdan, Quantization of Lie bialgebras, V: Quantum vertex operator algebras,
 {\em Selecta Math.  (N.S.)} {\bf 6} (2000), 105--130.

   \bibitem[FHL]{FHL}
I.~Frenkel, Y.~Huang, and J.~Lepowsky,  {\em On Axiomatic Approaches to Vertex Operator Algebras and
  Modules},  Memoirs of the Amer. Math. Soc., Vol. 104, 1993.

  \bibitem[FLM]{FLM}
I.~Frenkel, J.~Lepowsky, and A.~Meurman,
{\em {Vertex Operator Algebras and the Monster}},  {\em
  Pure and Applied Mathematics}, Volume 134, Academic Press, New York, 1988.

 \bibitem[G]{Garland}
H. Garland,  The arithmetic theory of loop algebras, {\em J. Algebra}
{\bf 53} (1978), 480--551.

\bibitem[I]{Io}
 K. Iohara, Bosonic representations of Yangian double, {\em J. Phys. A} {\bf 29} (1996), 4593.

\bibitem[JKLT1]{JKLT-21}
N.~Jing, F.~Kong, H.-S.~Li, and S.~Tan,
$({G},\chi_{\phi})$-equivariant $\phi$-coordinated quasi modules for nonlocal vertex algebras,
{\em J. Algebra} {\bf 570} (2021), 24-74.

\bibitem[JKLT2]{JKLT-22}
N.~Jing, F.~Kong, H.-S.~Li, and S.~Tan,
Deforming vertex algebras by vertex bialgebras,
{\em Commun. Contemp. Math.}, {\bf 26} (2024), 2250067.

\bibitem[JKLT3]{JKLT-23}
N.~Jing, F.~Kong, H.-S.~Li, and S.~Tan,
Twisted quantum affine algebras and equivariant $\phi$-coordinated modules for quantum vertex algebras,
arXiv: 2212.01895 [math.QA]; submitted for publication.

\bibitem[Kac]{Kac}
V. Kac, {\em Infinite dimensional Lie algebras,} Cambridge University Press, 1994.

\bibitem[Ka]{Ka}
C.~Kassel, {\em Quantum Groups,} GTM {\bf 155},
Springer-Verlag, New York, 1995.

\bibitem[Kh]{Kh}
 S. M. Khoroshkin, Central extension of the Yangian double,
{\em {Alg\`ebre non commutative, groupes quantiques et invariants
  ({R}eims, 1995)}} (1997), 119--135.

\bibitem[Ko]{Kong}
F.~Kong,  Quantum lattice vertex operator algebras,  {\em J. Pure Applied Algebra} {\bf 229} (2025), 107832.

\bibitem[KL]{KL1}
F.~Kong and H.-S.~Li, Double Yangians and quantum vertex algebras, I, preprint.

\bibitem[LL]{LL}
J.~Lepowsky and H.-S.~Li, {\em Introduction to Vertex Operator Algebras and Their
  Representations}, Prog. Math. Volume 227, Birkh\"{a}user Boston, 2004.

\bibitem[Li1]{Li-g1}
H.-S.~Li, Axiomatic $G_{1}$-vertex algebras, {\em Commun. Contemp. Math.} {\bf 5} (2003), 1--47.

\bibitem[Li2]{Li-nondeg}
H.-S. ~Li, Simple vertex operator algebras are nondegenerate, {\em J. Algebra} {\bf 267} (2003), 199--211.

\bibitem[Li3]{Li-nonlocal}
H.-S. ~Li, Nonlocal vertex algebras generated by formal vertex operators,
{\em Selecta Math.  (N.S.)} {\bf 11} (2005), 349--397.

\bibitem[Li4]{Li-qva2}
H.-S. ~Li, Constructing quantum vertex algebras, {\em Int. J. Math.} {\bf 17} (04) (2006), 441--476.

\bibitem[Li5]{Li-smash}
H.-S.~Li, A smash product construction of nonlocal vertex algebras,
{\em Commun. Contemp. Math.} {\bf 9} (2007), 605--637.

\bibitem[Li6]{Li-h-adic}
H.-S.~Li, $\hbar$-adic quantum vertex algebras and their modules,
{\em Commun. Math. Phys.} {\bf 296} (2010), 475--523.

\bibitem[LW]{LW}
H.-S. Li and Q. Wang, On vertex algebras and their modules associated with even lattices,
 {\em J. Pure Applied Algebra} {\bf 213} (2009), 1097\ndash 1111.

\bibitem[LX]{LX}
H.-S. Li and X. Xu, A characterization of vertex algebras
associated to even lattices, {\em J. Algebra} {\bf 173} (1995), 253-270.


\end{thebibliography}
\end{document}